\documentclass[12pt]{article}
\usepackage[english]{babel}
\usepackage[utf8]{inputenc}
\usepackage[T1]{fontenc}
\usepackage{tikz}
\usepackage{graphicx}
	\graphicspath{{images/}}
\usepackage{url}
\usepackage{enumitem}
\usepackage[all]{xy}
\usepackage[colorlinks,citecolor=blue,urlcolor=blue]{hyperref}
\usepackage{makeidx}
\makeindex
\usepackage{amsmath,amsfonts,amsthm,amssymb}
\usepackage{typearea}
\usepackage{graphicx}
\usepackage{geometry}
\newcommand{\stsets}[1]{\mathbb{#1}}
\newcommand{\R}{\stsets{R}}

\newcommand{\N}{\stsets{N}}

\theoremstyle{plain}

\newtheorem{theorem}{Theorem}[section]

\newtheorem{lemma}[theorem]{Lemma}
\newtheorem{assumption}[theorem]{Assumption}
\newtheorem{definition}[theorem]{Definition}
\theoremstyle{remark}
\newtheorem{example}[theorem]{Example}

\renewcommand{\P}{\mathbf{P}}
\DeclareMathOperator{\E}{{\bf E}}
\DeclareMathOperator{\one}{{ 1\hspace*{-0.55ex}I}}
\DeclareMathOperator{\card}{card}
\DeclareMathOperator{\Supp}{Supp}
\DeclareMathOperator*{\argmax}{arg\,max}
\DeclareMathOperator{\TLLN}{TLLN}
\DeclareMathOperator{\Cset}{\mathcal{C}}
\newcommand*\diff{\mathop{}\!\mathrm{d}}

\newcommand{\eps}{\varepsilon}

\renewcommand{\P}{\mathbf{P}}

\renewcommand{\epsilon}{\varepsilon}

\renewcommand{\phi}{\varphi}

\newlength{\querylen}
\usepackage{fancybox}

\def\acknowledgementsname{Acknowledgments}
\newenvironment{acks}[1][\acknowledgementsname]{\section*{#1}}{\par}
 
\if@balayout
    \renewenvironment{acks}[1][\acknowledgementsname]%
        {%
            \vskip0.5\baselineskip
            \small
            {\noindent\normalfont\sffamily\bfseries\acknowledgementsname}\par
            \begingroup\parindent 0pt\parskip 0.5\baselineskip
        }%
        {\endgroup}
\fi
\def\fundingname{Funding}
\newenvironment{funding}[1][\fundingname]{\begin{acks}[\fundingname]}{\end{acks}}

\title{Replica-mean-field limit of continuous-time fragmentation-interaction-aggregation processes}
\author{Michel Davydov \thanks{Division of Applied Mathematics, Brown University, Providence, RI, USA, \href{mailto:michel_davydov@brown.edu}{michel\_davydov@brown.edu}} }
\date{}
\begin{document}

\maketitle
\begin{abstract}
Many phenomena can be modeled as network dynamics with punctuate interactions. However, most relevant dynamics do not allow for computational tractability. To circumvent this difficulty, the Poisson Hypothesis regime replaces interaction times between nodes by independent Poisson processes, allowing for tractability in several cases, such as intensity-based models from computational neuroscience. This hypothesis is usually only conjectured, or numerically validated. In this work, we introduce a class of processes in continuous time called continuous fragmentation-interaction-aggregation processes, by analogy with previously introduced processes in discrete time. The state of each node, described by the stochastic intensity of an associated point process, aggregates arrivals from its neighbors and is fragmented upon departure. We consider the replica-mean-field version of such a process, which is a physical system consisting of randomly interacting copies of the network of interest. Generalizing results proved in discrete time and in the particular case of excitatory intensity-based neural dynamics, we prove that the Poisson Hypothesis holds at the limit of an infinite number of replicas.
\end{abstract}
\section{Introduction}

\subsection*{Motivation}
Many phenomena of interest, in the natural sciences or elsewhere, can be modeled as punctuate interactions between agents on an underlying network. Whether it be neural computations \cite{Galves_2013} \cite{Baccelli_2019}, opinion dynamics \cite{Amblard_2004}, epidemics propagation \cite{Pastor_Satorras_2015} or wireless communications \cite{Abishek17}, a natural way to model the evolution in time of a population of agents or nodes is to consider the times at which interactions happen at each node as the realization of a point process on the real line. The phenomena are then described through a system of stochastic differential equations satisfied by the stochastic intensities of the point processes. From the point of view of a single node, its state evolves in the following fashion: it aggregates arrivals from its neighbors, possibly with weights (these weights can be signed, for example, to model excitatory or inhibitory inputs in neural models), and in the event of a departure, its state is updated accordingly (it can for example decrease by one if we are interested in modeling services in queueing models, be reset to a resting state to mimic spiking in neural networks, be divided by two to model cellular division). In addition, we allow for the presence of a continuous drift, for example to model a refractory period after spiking in neuroscience models.

This flexibility allows for an accurate description of a wide variety of phenomena, but this accuracy comes at a price, namely, tractability, as neither the equations themselves, nor their associated functionals such as moment generating functions, admit closed forms except for some very particular cases. 

As such, a common approach is to simplify the model by neglecting certain characteristics of the phenomenon, such as considering agents to be independent, and/or considering a particular scaling of the system, typically removing finite-size effects. One classical simplification is the mean-field regime, obtained by letting the number of agents go to infinity and scaling the interactions accordingly, usually inversely proportionally to the number of agents, thus averaging interactions across the system \cite{Fournier_Locherbach_2016} \cite{LeBoudec07}. The resulting equation, common to all agents and usually of McKean-Vlasov type, often allows for closed forms to be obtained, at the cost of losing correlations between particles and the geometry of the underlying network.

In recent years, different approaches have been developed to circumvent these limitations. To incorporate heterogeneity, the properties of graphons (large dense graphs) have been used to derive new limit equations \cite{ZAN_2022} \cite{jabin2022}. In this setting, the limit object is an infinite system of ODEs. Another approach circumvents mean-field models altogether, relying instead on conditional independence properties and local weak limits to obtain local convergence \cite{Ramanan_2020}.

Another approach to obtain closed forms is called the Poisson Hypothesis. Popularized by Kleinrock for large queueing systems \cite{Kleinrock_1}, it states that the flow of arrivals to a given node can be approximated by a Poisson flow with rate equal to the average rate of the original flow of arrivals. In agent-based models, the flow of arrivals corresponds to the effect of interactions on a given node. Under the Poisson Hypothesis, the behavior of each agent is still described by a stochastic differential equation, but the agents are considered independent and interaction times are replaced by Poisson process clocks, which in certain models allows for tractability. This regime has been studied for queueing models by Rybko, Shlosman and others \cite{Vladimirov_2018} and by Baccelli and Taillefumier for intensity-based models from computational neuroscience \cite{Baccelli_2019}. 

A point of interest is the construction of physical models that, when properly scaled, converge to the Poisson Hypothesis regime, analogously to the classical mean-field construction. The replica-mean-field construction has been shown to be a successful answer to this question in various settings. This approach consists in building a new physical system comprising randomly interacting copies of the original network, and then letting the number of copies, or replicas, to go to infinity. From the point of view of a single node inside one of the replicas, its state evolves in a similar fashion to that of the original model, except that arrivals are now aggregated across neighbors in all the replicas. When there is a departure from the node, for each of its neighbors, an independent routing gives the index of the replica in which the neighbor will aggregate the arrival to its state. 

As the probability of two nodes interacting scales inversely proportionally to the number of replicas, the replicas become asymptotically independent when their number goes to infinity. Additionally, the aggregated arrivals to a given node can be seen as a (random) sum of rare events, which heuristically gives rise to a Poisson process at the limit in the number of replicas. These observations give an informal idea of how  it has been shown in the particular case of intensity-based neural dynamics that when the number of replicas goes to infinity, the dynamics of a typical replica converge to those of under the Poisson Hypothesis. This has been proved in both continuous and discrete time for excitatory neural dynamics \cite{bdt_2022}\cite{Davydov2022}. In the discrete time framework, a class of discrete-time processes, called fragmentation-aggregation-interaction processes, or FIAPs for short, for which the same limit theorem holds, has been introduced \cite{bdt_2022}. A natural question, and the aim of this work, is to introduce an analogous class of processes in continuous time and to extend the convergence result for replica-mean-field versions of such processes obtained in the specific case of excitatory neural dynamics.

The rest of the work is organized as follows: hereafter, we introduce the class of continuous-time fragmentation-interaction-aggregation processes and state the main result. Section 2 pertains to the proof of the result. Section 3 establishes a link between continuous-time and discrete-time FIAPs in the particular case of excitatory neural dynamics. 
\subsection*{Continuous-time fragmentation-interaction-aggregation processes}

Let $p>0.$ Given a point process $\mathbf{N}=(N_1,\ldots,N_p)$ on $\R^p,$ we introduce the network history  $(\mathcal{F}_t)_{t\in \mathbb{R}}$ as an increasing collection of $\sigma$-fields such that
\begin{equation*}
\mathcal{F}_t^\mathbf{N}=\{\sigma(N_1(B_1),...,N_K(B_K))|B_i\in \mathbf{B}(\mathbb{R}), B_i \subset (-\infty,t]\}\subset \mathcal{F}_t,
\end{equation*}
where $\mathcal{F}_t^\mathbf{N}$ is the internal history of the process $\mathbf{N}$.

For $1\leq i\leq p$, the $\mathcal{F}_t$-stochastic intensity $\{\lambda_i(t)\}_{t \in \mathbb{R}}$ of the associated point process $N_i$ is the $\mathcal{F}_t$-predictable random process satisfying for all $s<t \in \mathbb{R}:$
\begin{equation}
\label{eq_stoch_int_cfiap}
\E\left[N_i(s,t]|\mathcal{F}_s\right]=\E\left[\int_s^t \lambda_i(u)\diff u\big|\mathcal{F}_s\right].
\end{equation}
We will hereafter refer to \eqref{eq_stoch_int_cfiap} as \textit{the stochastic intensity property}. See \cite{Bremaud20} for more details on point processes admitting stochastic intensities.

Fragmentation-aggregation-interaction processes, or FIAPs for short, are processes with specific stochastic intensities, characterized below:
\begin{definition}
\label{def_cFIAP}
Let $K\geq 2.$ We define a \textit{continuous-time fragmentation-interaction-aggregation process} with parameter functions $(h_{j\rightarrow i})_{1\leq i\leq K, j \neq i}$, $(g_i)_{1\leq i\leq K}$ and $f$, hereafter referred to as a cFIAP$(f,g,h)$, as a collection of point processes $(N_i)_{1\leq i\leq K}$ on $\R$, hereafter referred to as the \textit{state processes}, where each $N_i$ admits a stochastic intensity $\lambda_i$ with respect to the network history such that for any $t\in \R^+,$ there exist
\begin{itemize}
    \item[$\bullet$]point processes $(\widehat{N}_{j\rightarrow i})_{1 \leq i \leq K, j \neq i}$ admitting stochastic intensities $(\widehat{\lambda}_{j \rightarrow i})$ hereafter referred to as \textit{interaction processes};
    \item[$\bullet$] measurable functions $(h_{j\rightarrow i})_{1\leq i\leq K, j \neq i}:\R\mapsto \R$ hereafter referred to as \textit{interaction functions} such that there exists $H>0$ that satisfies $|h_{j\rightarrow i}(t)|\leq H$ for all $i,j$ and all $t \in \R;$
    \item[$\bullet$]measurable functions $(g_i)_{1\leq i\leq K}:\R \mapsto \R^+$ and $(\sigma_i)_{1\leq i\leq K}:\R \mapsto \R^+$hereafter referred to as \textit{autonomous evolution functions};
    \item[$\bullet$]a Lipschitz function $f:\R\mapsto \R^+$ such that $f(0)=0$; 
\end{itemize}
such that for all $1\leq i\leq K$ and $t\in \R^+,$
\begin{equation}
\label{eq_cFIAP}
\begin{split}
\lambda_i(t)&=\lambda_i(0)+f\left(\sum_{j\neq i}\int_0^t h_{j\rightarrow i}(s)\widehat{N}_{j_\rightarrow i}(\diff s)\right)+\int_0^t(g_i(s,\lambda_i(s))-\lambda_i(s))N_i(\diff s)\\
&+\int_0^t(\sigma_i(s,\lambda_i(s))-\lambda_i(s))\diff s.
\end{split}
\end{equation}
\end{definition}
We will omit the parameters $f,g,h$ in the notation for cFIAPs when there is no ambiguity.
The dynamics \eqref{eq_cFIAP} reproduce in continuous time the dynamics of FIAPs defined in discrete time in earlier work \cite{bdt_2022}. The first term represents the \textit{agregation} of \textit{interactions} from the other nodes to the state of the $i-$th node, whereas the two last terms represent the autonomous evolution of the node as well as the change in its state when it is \textit{fragmented}.

We now formalize the Poisson Hypothesis.
\begin{definition}
\label{def_cFIAP_PH}
We say that a cFIAP satisfies the Poisson Hypothesis if all interaction times, namely all points of interaction processes defined in Definition \ref{def_cFIAP}, are given by independent Poisson processes. We denote with tildes all state processes in this regime. Namely, for all $i,j \in 1\leq i\leq K$ with $i \neq j, \hat{N}_{j\rightarrow i}$ are independent Poisson processes with intensities $s\rightarrow\E[\tilde{\lambda}_j(s)]$ where $\tilde{\lambda}_j$ is the stochastic intensity of the process $\tilde{N}_j$ with respect to the network history and for all $t\in \R^+$,
\begin{equation}
\label{eq_cFIAP_PH}
\begin{split}
\widetilde{\lambda}_i(t)&=\widetilde{\lambda}_i(0)+f\left(\sum_{j\neq i}\int_0^t h_{j\rightarrow i}(s)\hat{N}_{j\rightarrow i}(\diff s)\right)+\int_0^t(g_i(s,\widetilde{\lambda}_i(s))-\widetilde{\lambda}_i(s))\widetilde{N}_i(\diff s)\\
&+\int_0^t(\sigma_i(s,\widetilde{\lambda}_i(s))-\widetilde{\lambda}_i(s))\diff s.
\end{split}
\end{equation}
\end{definition}

Given a cFIAP, we now aim to define its replica-mean-field (RMF) version.

\begin{definition}
\label{def_cFIAP_RMF}
Let $K,M\geq 2.$ The $M-$replica-mean-field cFIAP$(f,g,h)$ is given by the collection of point processes $(N^M_{m,i})_{1\leq i\leq K, 1\leq m\leq M}$ admitting stochastic intensities $(\lambda^M_{m,i})$ such that for any $t\in \R^+,$ there exist
\begin{itemize}
    \item[$\bullet$]point processes $(\widehat{N}_{n,j \rightarrow i})_{1 \leq i \leq K, j \neq i, 1 \leq n \leq M}$ admitting stochastic intensities $(\widehat{\lambda}_{n,j,i})$ hereafter referred to as \textit{aggregation processes};
    \item[$\bullet$]a Lipschitz function $f:\R\mapsto \R^+$ such that $f(0)=0$;
    \item[$\bullet$] $(\mathcal{F}_t)$-predictable \textit{routing} processes $\{V^M_{(m,i)\rightarrow j}(t)\}_{t\in \mathbb{R}}$  for $1 \leq m \leq M, 1 \leq i,j \leq K,$ such that for each interaction time $T^M_{m,i}$, i.e., each point of the process $\hat{N}^M_{(m,i)\rightarrow j}$, the real-valued random variables $\{V^M_{(m,i)\rightarrow j}(T^M_{m,i})\}_j$ are mutually independent, independent from the past, i.e. from $\mathcal{F}_s$ for $s<T^M_{m,i},$ and uniformly distributed on $\{1,...,M\}\setminus\{m\}$
\end{itemize}
such that for all $1\leq m\leq M, 1\leq i\leq K$ and all $t\in \R^+,$
\begin{equation}
\label{eq_RMF_cFIAP}
\begin{split}
\lambda^M_{m,i}(t)&=\lambda^M_{m,i}(0)+f\left(\sum_{j\neq i}\sum_{n \neq m}\int_0^t h_{j\rightarrow i}(s)\one_{\{V^M_{(n,j)\rightarrow i}(s)=m\}}\hat{N}^M_{n,j\rightarrow i}(\diff s)\right)\\
&+\int_0^t(g_i(s,\lambda^M_{m,i}(s))-\lambda^M_{m,i}(s))N^M_{m,i}(\diff s)+\int_0^t(\sigma_i(s,\lambda^M_{m,i}(s))-\lambda^M_{m,i}(s))\diff s.
\end{split}
\end{equation}
\end{definition}

Hereafter, we will usually omit the superscript $M$ when referring to the RMF dynamics. 

Moreover, we will always assume that $\hat{N}^M_{n,j\rightarrow i}=N^M_{n,j}$ for all $1\leq n\leq M, 1\leq j\leq K$. We have defined it in all generality here for coherence with the definition of the Poisson Hypothesis.

We will moreover always consider the following assumptions on the functions $g_i$ and $\sigma_i$:

\begin{assumption}
\label{Ass_main}
For all $s,t \in \R,$ for all $i \in \{1,\ldots,K\},$ 
\begin{equation*}
g_i(s,t)\leq t \textrm{ and } \sigma_i(s,t)\leq t.
\end{equation*}
\end{assumption}

In particular, we have
\begin{equation*}
g_i(s,\lambda_i(s))\leq \lambda_i(s) \textrm{ and } \sigma_i(s,\lambda_i(s))\leq \lambda_i(s).
\end{equation*}
Note that this implies that the state processes $\lambda_i$ are always decreasing in between aggregations.

In a similar fashion to \cite{Davydov2022}, we also require the following assumption on the initial conditions:
\begin{assumption}
\label{Ass_2_cFIAP}
There exists $\xi_0>0$ such that for all $1\leq m \leq M, 1 \leq i \leq K$ and all $0<\xi\leq \xi_0$, $\E[e^{\xi\lambda_{m,i}(0)}]<\infty$.
\end{assumption}

\subsection*{Examples of continuous-time FIAPs}

We now give a few instances of specific FIAPs.

\begin{example}
\label{ex_GL}
Taking for $1\leq i,j\leq K$ and all $t\in \R,$\\
$h_{j\rightarrow i}(t)=\mu_{j\rightarrow i}\geq 0, f(t)=|t|, g_i(t, \lambda_i(t))=r_i>0, \sigma_i(t,\lambda_i(t))=b_i>0,$ we retrieve the excitatory Galves-Löcherbach model  for neural dynamics \cite{Galves_2013} \cite{Baccelli_2019}.
\end{example}
\begin{example}
\label{ex_GL_inh}
Taking for $1\leq i,j\leq K$ and all $t\geq 0,$\\
$h_{j\rightarrow i}(t)=\mu_{j\rightarrow i}\in \R, f(t)=\max(0,|t|), g_i(t, \lambda_i(t))=r_i>0, \sigma_i(t,\lambda_i(t))=b_i>0,$ we obtain a more general Galves-Löcherbach model for neural dynamics incorporating inhibition.
\end{example}
\begin{example}
\label{ex_GN}
Taking for $1\leq i,j\leq K$ and all $t\geq 0,$ \\
$h_{j\rightarrow i}(t)=\one_{\{j=i+1 \mod K\}}, f(t)=|t|, g_i(t, \lambda_i(t))=\lambda_i(t)-1, \sigma_i(t)=\lambda_i(t),$ we obtain a continuous-time concatenation queueing network. Note that when the Poisson Hypothesis holds, such a network is an instance of a Gordon-Newell queueing network \cite{Kleinrock_2}.
\end{example}
\subsection*{The main result}
Recall the following definition of convergence in total variation:

\begin{definition}
\label{def_tv}
Let $P$ and $Q$ be two probability measures on a probability space $(\Omega, \mathcal{F})$. We define the total variation distance by
\begin{equation*}
d_{TV}(P,Q)=\sup_{A\in \mathcal{F}}|P(A)-Q(A)|.
\end{equation*}
When $\Omega$ is countable, an equivalent definition is
\begin{equation*}
d_{TV}(P,Q)=\frac{1}{2}\sum_{\omega \in \Omega}|P(\omega)-Q(\omega)|.
\end{equation*}
A sequence of random variables is said to converge in total variation when the corresponding sequence of their distributions does.
\end{definition}

The following theorem is the main result of this work:

\begin{theorem}
\label{th_Poisson_CV_cFIAP}
There exists $T\in \R^+$ such that for $t \in [0,T]$, if
\begin{equation*}
A^M_{m,i}(t)=\sum_{j \neq i} \sum_{n \neq m}  \int_0^t h_{j\rightarrow i}(s) \one_{\{V^M_{(n,j)\rightarrow i}(s)=m\}} N^M_{n,j}(\diff s),
\end{equation*}
with $N^M$ defined in \eqref{eq_RMF_cFIAP},
and 
\begin{equation*}
\tilde{A}_i(t)=\sum_{j \neq i}\int_0^t h_{j\rightarrow i}(s)\widehat{N}_{j\rightarrow i}(\diff s),
\end{equation*}
with $(\widehat{N}_{j \rightarrow i})_j$ independent Poisson point processes with respective intensities $s\mapsto\E[\tilde{\lambda}_j(s)]$, then, 
\begin{enumerate}
\item the processes $(\tilde{A}_1,\ldots,\tilde{A}_K)$ are independent, as are the processes $(\tilde{\lambda}_1,\ldots,\tilde{\lambda}_K);$
\item there exists $C_1>0$ such that for all $(m,i) \in \{1,\ldots,M\}\times\{1,\ldots,K\}$,
\begin{equation*}
    d_{TV}(A^M_{m,i}(t),\tilde{A}_i(t))\leq \frac{C_1}{\sqrt{M}};
\end{equation*}
\item there exists $C_2>0$ such that for all $(m,i) \in \{1,\ldots,M\}\times\{1,\ldots,K\}$,
\begin{equation*}
    d_{TV}(\lambda^M_{m,i}(t),\Tilde{\lambda}_i(t))\leq \frac{C_2}{\sqrt{M}},
\end{equation*}
where $\lambda^M_{m,i}(t)$ is defined by \eqref{eq_RMF_cFIAP} and $\Tilde{\lambda}_i(t)$ is defined in \eqref{eq_cFIAP_PH};
\item let $\mathcal{N}$ be a finite subset of $\N^*$, for all $i\in \{1,\ldots,K\},$ the processes $(A^M_{m,i}(\cdot))_{m \in \mathcal{N}}$ and $(\lambda^M_{m,i}(\cdot))_{m \in \mathcal{N}}$ weakly converge in the Skorokhod space $D([0,T])^{\card(\mathcal{N})}$ endowed with the product Skorokhod topology to $\card(\mathcal{N})$ independent copies of the corresponding limit processes $(\tilde{A}_i(\cdot))$ and $(\tilde{\lambda}_i(\cdot))$ when $M \rightarrow \infty$.
\end{enumerate}

\end{theorem}

\section{Proof of the theorem}

We will follow the general proof framework developed by the author in the previous work \cite{Davydov2022}. We will emphasize the technical points that were adapted to this more general case.

First, we remind the following Poisson embedding representation for point processes with a stochastic intensity \cite[Section 3]{BremMass96}, which allows us to couple all the state processes defined by \eqref{eq_RMF_cFIAP} through their Poisson embeddings and initial conditions.
\begin{lemma}
\label{lem_poisson_embedding}
Let $N$ be a point process on $\mathbb{R}$. Let $(\mathcal{F}_t)$ be the internal history of $N$. Suppose $N$ admits a $(\mathcal{F}_t)$-stochastic intensity $\{\mu(t)\}_{t \in \mathbb{R}}.$ Then there exists a Poisson point process $\overline{N}$ with intensity 1 on $\mathbb{R}^2$ such that for all $C \in \mathcal{B}(\mathbb{R}),$
\begin{equation*}
    N(C)=\int_{C \times \mathbb{R}}\one_{[0, \mu(s)]}(u)\overline{N}(\diff s\times \diff u).
\end{equation*}
\end{lemma}

For $m \geq 1, M \geq 1, 1 \leq i \leq K,$ let $\overline{N}_{m,i}$ be i.i.d. Poisson point processes on $\R^+ \times \R^+$ with intensity 1.

Let $\Omega=(\mathbb{R}^+ \times ((\mathbb{R}^+)^2)^{\mathbb{N}^*})^{\mathbb{N}^*}$ be a probability space endowed with the probability measure $(\mu_0 \otimes P)^{\otimes \mathbb{N}^*},$ where $\mu_0$ is the law of the initial conditions and $P$ is the law of a Poisson process with intensity 1 on $(\mathbb{R}^+)^2$.
We construct on $\Omega$ the following processes:
\begin{itemize}
    \item The processes $(N^M_{m,i}(t)), m \geq 1, M \geq 1, 1 \leq i \leq K,$ with stochastic intensities $(\lambda^M_{m,i}(t))$ satisfying
    \begin{equation}
    \label{eq_rmf_coupling_cFIAP}
    \begin{split}
    &\lambda^M_{m,i}(t)=\int_0^t \int_0^{+\infty}\left(g_i(s,\lambda^M_{m,i}(s))-\lambda^M_{m,i}(s)\right) \one_{[0,\lambda^M_{m,i}(s)]}(u) \overline{N}_{m,i}(\diff s \times \diff u)\\
    &+f\left(\sum_{n \neq m} \sum_{j \neq i} \int_0^t\int_0^{+\infty} h_{j\rightarrow i}(s)\one_{\{V^M_{(n,j)\rightarrow i}(s)=m\}} \one_{[0,\lambda^M_{n,j\rightarrow i}(s)]}(u) \overline{N}_{n,j\rightarrow i}(\diff s \times \diff u)\right)\\
    &+\int_0^t(\sigma_i(s,\lambda^M_{m,i}(s))-\lambda^M_{m,i}(s))\diff s+\lambda^M_{m,i}(0),
    \end{split}
    \end{equation}
    with $\lambda^M_{m,i}(0)=Z_{i}$ for all $m \in \N^*$, where $Z_i$ is a random variable with law $\mu_0$ and where, for all $M$, $(V^M_{(n,j)\rightarrow i}(t))_j$ are càdlàg stochastic processes such that for each point $T_{n,j}$ of $\hat{N}_{n,j}$, the random variables $(V^M_{(n,j)\rightarrow i}(T_{n,j}))_j$ are independent of the past, mutually independent and uniformly distributed on $\{1,...,M\}\setminus \{n\}$, considered as marks of the Poisson point process $\overline{N}_{n,j}$. Namely, to each point of the Poisson embedding, we attach a mark that is an element of $(\N^{K})^{N^*}$, where the $M$th term of the sequence corresponds to $(V^M_{(n,j)\rightarrow i}(T_{n,j}))_j$.
    \item The processes $(\tilde{N}_{i}(t)), 1 \leq i \leq K,$ with stochastic intensities $(\tilde{\lambda}_{i}(t))$ verifying
    \begin{equation}
    \label{eq_limit_coupling_cFIAP}
    \begin{split}
    \widetilde{\lambda}_{i}(t)&=\widetilde{\lambda}_{i}(0)+f\left(\sum_{j \neq i} \int_0^t\int_0^{+\infty}h_{j\rightarrow i}(s)\one_{[0,\E[\widetilde{\lambda}_{j}(s)]]}(u)\overline{N}_{j,i}(\diff s \times \diff u)\right)\\
    &+\int_0^t \int_0^{+\infty}\left(g_i(s,\widetilde{\lambda}_i(s))-\widetilde{\lambda}_{i}(s)\right) \one_{[0,\widetilde{\lambda}_{i}(s)]}(u)\overline{N}_{i,i}(\diff s \times \diff u),
    \end{split}
    \end{equation}
    with $\widetilde{\lambda}_{i}(0)=Z_{i}$.
\end{itemize}

Just as in the particular case of neural dynamics from Example \ref{ex_GL} covered in \cite{Davydov2022}, this representation is sufficient to derive the following, which is statement 1 of Theorem \ref{th_Poisson_CV_cFIAP}.
\begin{lemma}
\label{lem_asympt_indep_cFIAP}
The processes $(\tilde{A}_i)_{1 \leq i\leq K}$ are independent, as are the processes $(\tilde{\lambda}_1,\ldots,\tilde{\lambda}_K)$.
\end{lemma}
\begin{proof}
For all $t \in [0,T]$, we can write using the construction above
\begin{equation*}
\tilde{A}_{i}(t)=\sum_{j\neq i}\int_0^t\int_0^{+\infty}h_{j\rightarrow i}(s)\one_{[0,\E[\tilde{\lambda}_{j}(s)]]}(u)\overline{N}_{j,i}(\diff s \times \diff u).
\end{equation*}

Therefore, all the randomness in $\tilde{A}_i$ is contained in the Poisson embeddings $(\hat{N}_{k,i})_{1 \leq k \leq K}$. Thus, for $i \neq j$, $\tilde{A}_i$ and $\tilde{A}_j$ are independent. The independence of the processes $(\tilde{\lambda}_1,\ldots,\tilde{\lambda}_K)$ follows from a mapping theorem argument.
\end{proof}
\subsection*{Properties of the replica-mean-field and limit processes}
\label{subsec_RMF_properties_cFIAP}
In this section, we prove several properties of the RMF and limit dynamics that will be used throughout the proof. In what follows, we will often omit the $M$ superscript in the notations $N^M_{m,i}$, $A^M_{m,i}$ and $\lambda^M_{m,i}$ to increase readability.

We generalize from \cite[Lemma 2.3]{Davydov2022} the following representation of the arrival process $A_{m,i}$. For $n \neq m$ and $j\neq i$, if $S \in \Supp(N_{n,j}|_{[0,T)})$, we define $B^M_{S,(n,j)\rightarrow (m,i)}$ to be the random variable equal to 1 if the routing between replicas at time $S$ caused by a departure in node $j$ in replica $n$ chose the replica $m$ for the recipient $i$ of the interaction thus produced, and 0 otherwise. In other words, using notation from \eqref{eq_RMF_cFIAP}, $B^M_{S,(n,j)\rightarrow (m,i)}=\one_{V_{(n,j)\rightarrow(m,i)}(S)}.$ As such, it is clear that we can write for all $t \in [0,T], m \in \{1,\ldots,M\}$ and $i\in \{1,\ldots,K\},$ 
\begin{equation}
\label{eq_arr_rep_cFIAP}
A_{m,i}(t)=\sum_{n \neq m}\sum_{j \neq i}\sum_{k \in N_{n,j}\cap[0,t]}h_{j\rightarrow i}(k)B^M_{k,(n,j)\rightarrow (m,i)}.
\end{equation}
Note that when $m,n,i$ and $j$ are fixed, the random variables $(B^M_{k,(n,j)\rightarrow (m,i)})_{k \leq N_{n,j}([0,T])}$ are i.i.d. 
Also note that when $n,j,i$ and $k$ are fixed, the joint distribution of $(B^M_{k,(n,j)\rightarrow (m,i)})_m$ with $m \in \{1,\ldots,M\}$ is that of Bernoulli random variables with parameter $\frac{1}{M-1}$ such that exactly one of them is equal to 1, all the others being equal to 0.
Combining these two observations allows us to show that the following lemma, highlighting a key property of the replica-mean-field approach, holds:

\begin{lemma}
\label{lem_cond_indep_bern_cFIAP}
Fix $(m,i) \in \{1,\ldots,M\}\times\{1,\ldots,K\}$. Keeping notation from \eqref{eq_arr_rep_cFIAP}, let \\
$N=(N_{n,j}([0,t]))_{n \neq m, j \neq i} \in  \mathbb{N}^{(K-1)(M-1)}.$

Conditionally on the event $\{N=q\},$ for $q=(q_{n,j})_{n \neq m, j \neq i} \in \mathbb{N}^{(K-1)(M-1)},$ the random variables $(B^M_{k,(n,j)\rightarrow (m,i)})_{n \neq m, j \neq i, k \in \{1,\ldots,q_{n,j}\}}$ are independent Bernoulli random variables with parameter $\frac{1}{M-1}.$
\end{lemma}
\begin{proof}
The structure of the proof is unchanged from \cite{Davydov2022}: since $N$ is entirely determined by the Poisson embeddings $(\hat{N}_{n,j})_{j\neq i}$ and the arrivals to the nodes $(n,j)$ from all the nodes $h\neq j$ across replicas, it is sufficient to show that these arrivals and the routing variables $(B^M_{k,(n,j)\rightarrow (m,i)})_{k\leq\hat{N}_{n,j}([0,t]\times\R^+)}$ are independent. Intuitively, this holds because arrivals are aggregated across all replicas, which will erase the eventual dependencies due to the routing variables to nodes $i$ choosing one replica rather than another.\\
In order to transcribe this intuition rigorously, we first show that the total number of departures from nodes $i$ up to time $t$, that is, $\sum_{l=1}^M N_{l,i}([0,t])$, and the routing variables \\
$(B^M_{k,(n,j)\rightarrow (m,i)})_{k\leq\hat{N}_{n,j}([0,t]\times\R^+)}$ are independent. Indeed, using the representation given by Lemma \ref{lem_poisson_embedding}, due to the structure of the Poisson embeddings $(\hat{N}_{l,i})_{l\in \{1,\ldots,M\}}$, there is a point of $\sum_{l=1}^M N_{l,i}$ in some interval $I$ iff there is a point of the superposition of the corresponding Poisson embeddings such that the x-coordinate is in $I$ and the y-coordinate is under the curve of the function $t\rightarrow \sum_{l=1}^M\lambda_{l,i}(t)$. In turn, the last event does not depend on $(B^M_{k,(n,j)\rightarrow (m,i)})_{k\leq\hat{N}_{n,j}([0,t]\times\R^+)}$, as the symmetry inherent to the replica structure ensures that all arrivals increment $t\rightarrow \sum_{l=1}^M\lambda_{l,i}(t)$ by the same amount, which concludes the proof of this preliminary remark.

For all $(n,j)$ such that $n \neq m$ and $j \neq i$, let 
\begin{equation*}
A_{i \rightarrow (n,j)}(t)=\sum_{l\neq n}\sum_{k\in N_{l,i}\cap[0,t]}h_{j,i}(k)B^M_{k,(l,i)\rightarrow (n,j)}.
\end{equation*}
Note that $A_{i \rightarrow (n,j)}(t)$ represents the arrivals to node $j$ in replica $n$ from all nodes $i$ across replicas. As such, it is clear that we can write
\begin{equation*}
A_{i \rightarrow (n,j)}(t)=\sum_{k\in \sum_{l \neq n} N_{l,i}\cap[0,t]} h_{j,i}(k) B^M_{k,(i)\rightarrow (n,j)},
\end{equation*}
where $(B^M_{k,(i)\rightarrow (n,j)})$ are independent Bernoulli random variables with parameter $\frac{1}{M-1}$ such that they and $(B^M_{k,(n,j)\rightarrow (m,i)})$ are independent. Then, $A_{i \rightarrow (n,j)}(t)$ and $(B^M_{k,(n,j)\rightarrow (m,i)})$ are independent. Therefore, $N$, which is entirely determined by the Poisson embeddings $(\hat{N}_{n,j})$ and the arrivals $(A_{h \rightarrow (n,j)}(t))_{h\neq j}$, and
$(B^M_{k,(n,j)\rightarrow (m,i)})_{k\leq\hat{N}_{n,j}([0,T]\times\R^+)}$, are independent. Thus, conditioning on $N$ does not break independence between the variables $(B^M_{k,(n,j)\rightarrow (m,i)}).$
\end{proof}
We will now bound the moments of both the M-replica and limit processes, using the bounds on the moments of the initial conditions. The validity of this bound is the main reason for the introduction of Assumption \ref{Ass_main}, which allows to stochastically dominate the dynamics by the same dynamics without the autonomous evolution integral terms, which enables Grönwall's lemma.
\begin{lemma}
\label{lem_moment_bound_cFIAP}
Suppose the initial conditions $Z_i$ verify Assumption \ref{Ass_2_cFIAP}.
Then, for all $p\geq 1$, for all $(m,i) \in \{1,\ldots,M\}\times\{1,\ldots,K\}$, for all $t \in [0,T]$, there exists $Q_p \in \mathbb{R}_p[X]$, a polynomial of degree exactly $p$, such that
\begin{equation}
\label{eq_gronwall_p_cFIAP}
\E[\lambda_{m,i}^{p}(t)]\leq Q_p(\E[\lambda_{m,i}(0)]).
\end{equation}
\end{lemma}
\begin{proof}
Note that, by Assumption \ref{Ass_main} and monotonicity, the dynamics that we consider are stochastically dominated by the same dynamics without the autonomous evolution terms, by which we mean the two last integral terms in \eqref{eq_RMF_cFIAP}. Thus, we can restrict ourselves to this special case.
We first prove the result for $p=1$.
Let $t\in [0,T]$.
Let $i(t)=\argmax_{j \in \{1,\ldots,K\}}\left|\E[\widehat{\lambda}_{n,j \rightarrow i}(t)]\right|$.
Using the stochastic intensity property \eqref{eq_stoch_int_cfiap}, we have
\begin{equation*}
\E[\lambda_{m,i(t)}(t)]=\E[\lambda_{m,i(t)}(0)]+\E\left[f(\sum_{n \neq m} \sum_{j \neq i(t)} \int_0^t \one_{\{V_{(n,j)\rightarrow i(t)}(s)=m\}} \widehat{\lambda}_{n,j\rightarrow i(t)}(s)\diff s)\right].
\end{equation*}
Using the assumptions on $f$ and $h_{j\rightarrow i}$ from Definition \ref{def_cFIAP}, we have
\begin{equation*}
\E[\lambda_{m,i(t)}(t)]\leq\E[\lambda_{m,i(t)}(0)]+\frac{H}{M-1}\sum_{n \neq m} \sum_{j \neq i(t)} \E\left[\int_0^t\lambda_{n,j}(s) \diff s\right].
\end{equation*}
By the definition of $i(t)$ and exchangeability of the replicas, we have
\begin{equation*}
\E[\lambda_{m,i(t)}(t)]\leq\E[\lambda_{m,i(t)}(0)]+H \sum_{j \neq i(t)} \int_0^t\E\left[\lambda_{m,i(s)}(s)\right] \diff s.
\end{equation*}

This gives by Grönwall's lemma the desired result:
\begin{equation}
\label{eq_gronwall_1_cFIAP}
\E[\lambda_{m,i(t)}(t)]\leq \E[\lambda_{m,i(t)}(0)]e^{(K-1)HT}=:Q_1(\E[\lambda_{m,i(t)}(0)]).
\end{equation}

This reasoning can be extended by induction to all $p\geq 2$, establishing a stochastic differential equation for $\lambda^{p+1}_{m,i(t)}(t)$ and repeating the steps above.
\end{proof}
Finally, note that the exact same reasoning can be applied to obtain an equivalent result for the limit process, which we will only state:
\begin{lemma}
\label{lem_limit_moment_bound_cFIAP}
For all $p\geq 1$, for all $i \in \{1,\ldots,K\}$, for all $t \in [0,T]$, there exists $\tilde{Q}_p \in \mathbb{R}_p[X]$ a polynomial of degree exactly $p$ such that
\begin{equation}
\label{eq_gronwall_tilde_p}
\E[\tilde{\lambda}_{i}^{p}(t)]\leq \tilde{Q}_p[\E[\tilde{\lambda}_{i}(0)]].
\end{equation}
\end{lemma}

Lemma \ref{lem_moment_bound_cFIAP} allows us to prove the following result, which states that Assumption \ref{Ass_2_cFIAP} can be propagated to any time $t$ less than some fixed $T.$

\begin{lemma}
\label{lem_exp_moment_bound_cFIAP}
There exists $\xi_0>0$ and $T>0$ such that for $\xi\leq \xi_0$ and all $t\leq T,$
\begin{equation}
\label{eq_exp_mom_bound_cFIAP}
\E[e^{\xi\lambda_{m,i}(t)}]<\infty \textrm{ and } \E[e^{\xi\tilde{\lambda}_{m,i}(t)}]<\infty.
\end{equation}
\end{lemma}
\begin{proof}
To prove this result, we once again consider exchangeable dynamics without the autonomous evolution terms, by which we once again mean the two last integral terms of \eqref{eq_RMF_cFIAP}, using the same observation as previously, namely that nonexchangeable dynamics with these terms are stochastically dominated by exchangeable dynamics without them, to generalize the result. Note in addition that exchangeable dynamics with interaction functions $h_{j\rightarrow i}$ are dominated by exchangeable dynamics with interaction functions $|h_{j\rightarrow i}|$.
Let $\xi_0$ be as in Assumption \ref{Ass_2_cFIAP}. Let $t\in [0,T].$
From \eqref{eq_RMF_cFIAP}, let us write out the equation verified by $e^{\xi\lambda_{m,i}(t)}$:

\begin{equation*}
e^{\xi\lambda_{m,i}(t)}=e^{\xi\lambda_{m,i}(0)}+f\bigg(\sum_{j\neq i}\sum_{n\neq m}\int_0^t\one_{\{V^M_{(n,j)\rightarrow i}(s)=m\}}e^{\xi\lambda_{m,i}(s)}(e^{\xi h_{j\rightarrow i}(s)}-1)N_{n,j}(\diff s)\bigg).
\end{equation*}

Taking the expectation, using the stochastic intensity property \eqref{eq_stoch_int_cfiap} and the conditions on $f$ and $(h_{j\rightarrow i})$, we get
\begin{equation*}
\E[e^{\xi\lambda_{m,i}(t)}]\leq \E[e^{\xi\lambda_{m,i}(0)}]+\frac{1}{M-1}\sum_{j\neq i}\sum_{n\neq m}\int_0^t\E[e^{\xi\lambda_{m,i}(s)}(e^{\xi H}-1)\lambda_{n,j}(s)]\diff s.
\end{equation*}
Using exchangeability between replicas, this boils down to
\begin{equation*}
\E[e^{\xi\lambda_{m,i}(t)}]=\E[e^{\xi\lambda_{m,i}(0)}]+\sum_{j\neq i}\int_0^t\E[e^{\xi\lambda_{m,i}(s)}(e^{\xi H}-1)\lambda_{m,j}(s)]\diff s.
\end{equation*}
Since we are looking at dynamics without resets and with only nonnegative interactions, $\lambda_{m,i}(s)$ and $\lambda_{m,j}(s)$ are positively correlated. Therefore, we have
\begin{equation*}
\E[e^{\xi\lambda_{m,i}(t)}]\leq \E[e^{\xi\lambda_{m,i}(0)}]+(e^{\xi\mu}-1)\sum_{j\neq i}\int_0^t\E[e^{\xi\lambda_{m,i}(s)}]\E[\lambda_{m,j}(s)]\diff s.
\end{equation*}
By Lemma \ref{lem_moment_bound_cFIAP} and Assumption \ref{Ass_2_cFIAP}, we have the existence of a constant $B>0$ such that
\begin{equation*}
\E[e^{\xi\lambda_{m,i}(t)}]\leq \E[e^{\xi\lambda_{m,i}(0)}]+(e^{\xi\mu}-1)(K-1)B\int_0^t\E[e^{\xi\lambda_{m,i}(s)}]\diff s.
\end{equation*}
The desired result follows from Grönwall's lemma. The equivalent result for $\tilde{\lambda}_{m,i}(t)$ is obtained in the same way.
\end{proof}

\subsection*{Poisson approximation bound}
\label{subsec_Poisson_bound_cFIAP}
The goal of this section is to extend the bound obtained using the Chen-Stein method \cite{Chen75} in \cite{Davydov2022} to obtain a bound in total variation distance between the arrivals term \eqref{eq_arr_rep_cFIAP} and the limit sum of Poisson random variables.
Recall that \eqref{eq_arr_rep_cFIAP} states that for all $t \in [0,T], m \in \{1,\ldots,M\}, i\in \{1,\ldots,K\},$ 
\begin{equation*}
A_{m,i}(t)=\sum_{j \neq i}A_{j\rightarrow (m,i)}(t),
\end{equation*}
where for all $j\neq i,$
\begin{equation*}
A_{j\rightarrow (m,i)}(t)=\sum_{n \neq m}\sum_{k \in \hat{N}_{n,j}\cap [0,t]}h_{j\rightarrow i}(k)B^M_{k,(n,j)\rightarrow (m,i)}.
\end{equation*}

As a temporary assumption to obtain a bound in total variation distance using the Chen-Stein method, that will be relaxed later by using a density argument, we will hereafter consider that the functions $h_{j\rightarrow i}$ are simple, that is, that they are finite linear combinations of indicator functions. More precisely, we suppose there exist $p\geq 0, a_1,\ldots,a_p\in \R$ and $A_1,\ldots,A_p$ measurable subsets of $\R$ such that for all $t\in \R,$
\begin{equation*} 
h_{j\rightarrow i}(t)=\sum_{l=0}^p a_l\one_{A_l}(t).
\end{equation*}
Therefore, we can write
\begin{equation}
\label{eq_arrivals_simple}
A_{j \rightarrow (m,i)}(t)=\sum_{l=0}^p a_l\sum_{n\neq m}\sum_{k\in \hat{N}_{n,j}\cap([0,t]\cap A_l)} B^M_{k,(n,j)\rightarrow(m,i)}.
\end{equation}

By the complete independence property of the Poisson point process, we see that without loss of generality, we can assume that $h_{j\rightarrow i}$ is a constant.

This allows us to now simply reuse the result proved by the author in \cite[Lemma 2.10]{Davydov2022}, which we recall below, using notation consistent with \eqref{eq_arr_rep_cFIAP}, with the addition of a multiplicative constant $C$ in the bound to account for the function $h_{j\rightarrow i}$. Since what follows is done with $t \in [0,T]$ fixed, we will additionally denote $N_{n,j}([0,t])$ by $N_{n,j}$, continuing to omit the $M$ superscript to simplify notation. 
\begin{lemma}
\label{lem_poisson_chenstein_bound_cFIAP}
Let $M>1$. Let $(m,i) \in \{1,\ldots,M\}\times \{1,\ldots,K\}$. For $j \in \{1,\ldots,K\}\setminus\{i\}$, let $A_{j \rightarrow (m,i)}=\sum_{n \neq m}\sum_{k\leq N_{n,j}}h_{j\rightarrow i}(k)B_{k, (n,j) \rightarrow (m,i)}$ with $h_{j\rightarrow i}$ simple, and let $\tilde{A}_{j \rightarrow i}$ be independent Poisson random variables with means $\E[N_{1,j}]$. Then, there exists a positive finite constant $C$ such that
\begin{equation}
\label{eq_chen_poisson_bound_cFIAP}
\begin{split}
d_{TV}(A_{j \rightarrow (m,i)},\tilde{A}_{j \rightarrow i}) &\leq C\Bigg(\bigg(1\wedge\frac{0.74}{\sqrt{\E[N_{1,j}]}}\bigg)\frac{1}{M-1}\E\left[\left|\sum_{n \neq m}\left(\E[N_{n,j}]-N_{n,j}\right)\right|\right]\\
&+\frac{1}{M-1}\left(1\wedge \frac{1}{\E[N_{1,j}]}\right)\E[N_{1,j}]\Bigg).
\end{split}
\end{equation}
\end{lemma}
We refer to \cite{Davydov2022} for the proof and comments on this result, which relies on the Chen-Stein method for Poisson approximation, using the conditional independence property proved in Lemma \ref{lem_cond_indep_bern_cFIAP} to condition on the random amount of aggregations.

Note that if the bound on the right-hand side goes to 0 when $M \rightarrow \infty,$ as we will endeavor to prove in the next section, then we will obtain convergence in total variation for all measurable functions $h_{j\rightarrow i},$ as any such function can be represented as a uniform limit of simple functions, and the uniform limit commutes with total variation convergence.

\subsection*{Decoupling arrivals and outputs: a fixed point scheme approach}
As we have seen in the Lemma \ref{lem_poisson_chenstein_bound_cFIAP}, for the Poisson approximation to hold, it is sufficient to prove a law of large numbers-type result on the random variables $(N_{n,j})_{n \neq m}$. However, since these random variables themselves depend on the random variables $(A_{j\rightarrow (m,i)})_{m \in \{1,\ldots,M\}}$, a direct proof seems difficult to obtain.

As such, we follow the approach of \cite{Davydov2022} and consider Equation \eqref{eq_RMF_cFIAP} as the fixed point equation of some function on the space of probability laws on the space of càdlàg trajectories. This fixed point exists and is necessarily unique due to the fact that Equation \eqref{eq_RMF_cFIAP} admits a unique solution. The main idea goes as follows: if we endow this space with a metric that makes it complete, in order to prove that the property of interest, namely that the aforementioned law of large numbers holds at the fixed point, it is sufficient to show that, on one hand, if this property holds for a given probability law, it also holds for its image by the function; and that on the other hand, the function's iterates form a Cauchy sequence. This approach is similar to the one developed in \cite{bdt_2022}, where propagation of chaos is proven in discrete time by showing that the one-step transition of the discrete dynamics preserves a triangular law of large numbers.

Our goal in this section is to prove the two aforementioned points. We start by introducing the metric space we will be considering and defining the function on it. Fix $T \in \R$, and let $D_T$ be the space of càdlàg functions on $[0,T]$ endowed with the Billingsley metric \cite{billingsley1968}: for $x,y \in D_T$, let $$d_{D_T}(x,y)=\inf_{\theta \in \Theta}\max(|||\theta|||,\|x-y\circ \theta\|),$$
where $$\Theta=\{\theta:[0,T]\rightarrow[0,T], \textrm{ s.t. } \theta(0)=0, \theta(T)=T, \textrm{ and } |||\theta|||<\infty \},$$ where $$|||\theta|||=\sup_{s\neq t \in [0,T]}\left|\log\left(\frac{\theta(t)-\theta(s)}{t-s}\right)\right|.$$ Intuitively, $\Theta$ represents all possible "reasonable" time shifts allowing one to minimize the effect of the jumps between the two functions $x$ and $y$, where "reasonable" means that all slopes of $\theta$ are close to 1.

We denote by $d_{D_T,U}$ the uniform metric on $D_T$: for $x,y \in D_T$, \begin{equation*}
d_{D_T,U}(x,y)=\|x-y\|.
\end{equation*} 
Note that we have for all $x,y \in D_T, d_{D_T}(x,y)\leq d_{D_T,U}(x,y)$, since the uniform metric corresponds precisely to the case where $\theta$ is the identity function.

Let $\mathcal{P}(D_T)$ be the space of probability measures on $D_T$. We endow it with the Kantorovitch metric \cite{Kanto42} (also known as the Wasserstein distance or the earth mover's distance): for $\mu,\nu \in \mathcal{P}(D_T)$, let 
\begin{equation*}
\mathcal{K}_T(\mu,\nu)=\inf_{\Pi \in D_T \times D_T}\E[ d_{D_T}(x,y)],
\end{equation*}
where $\Pi$ is a coupling s.t. $x \overset{\mathcal{L}}{=} \mu$ and $y\overset{\mathcal{L}}{=} \nu$. Finally, we fix $K,M \in \N$ and consider the space $(\mathcal{P}(D_T))^{MK}$ endowed with the 1-norm metric: for $\mu,\nu \in \mathcal{P}(D_T)$, let $$\mathcal{K}_T^{MK}(\mu,\nu)=\sum_{m=1}^M\sum_{i=1}^K \mathcal{K}_T(\mu_{m,i},\nu_{m,i}).$$

It is known that $(D_T,d_{D_T})$ is a complete separable metric space, see \cite{billingsley1968}, and thus
that $(\mathcal{P}(D_T),\mathcal{K}_T)$ and $(\mathcal{P}(D_T))^{MK},\mathcal{K}_T^{MK})$ are as well, see \cite{BogKol2012}.

We will also need to consider $\mathcal{P}(D_T)$ endowed with a Kantorovitch metric based on the uniform metric: we introduce for $\mu,\nu \in \mathcal{P}(D_T)$,  $$\mathcal{K}_{T,U}(\mu,\nu)=\inf_{\Pi \in D_T \times D_T} \E[d_{D_T,U}(x,y)],$$
where $\Pi$ is a coupling s.t. $x \overset{\mathcal{L}}{=} \mu$ and $y\overset{\mathcal{L}}{=} \nu$. We also introduce its product version $\mathcal{K}_{T,U}^{MK}$ defined analogously to above. Note that even though $(D_T,d_{D_T,U})$ is a complete metric space, it is not separable, therefore $(\mathcal{P}(D_T),\mathcal{K}_{T,U})$ is not a priori a complete metric space.

We now define the following mapping:

\begin{align*}
  \Phi \colon (\mathcal{P}(D_T))^{MK} &\to (\mathcal{P}(D_T))^{MK}\\
  \mathcal{L}(M)&\mapsto \Phi(\mathcal{L}(M)),
\end{align*}
\\
where $\mathcal{L}(M)$ is the law of $M$ and for all $(m,i)\in \{1,\ldots,M\}\times\{1,\ldots,K\}$, $\Phi(\mathcal{L}(M))$ is the law of the stochastic intensity $\lambda^\Phi_{m,i}$ of a point process $N^{\Phi}_{m,i}$ such that $\lambda^{\Phi}_{m,i}$ is the solution of the stochastic differential equation
\begin{equation}
\label{eq_input_output_map}
\begin{split}
\lambda^{\Phi}_{m,i}(t)&=\lambda^{\Phi}_{m,i}(0)+f(\sum_{j\neq i}\sum_{n \neq m}\int_0^t h_{j\rightarrow i}(s)\one_{\{V^M_{(n,j)\rightarrow i}(s)=m\}} M_{n,j\rightarrow i}(\diff s))\\
&+\int_0^t(g_i(s,\lambda^{\Phi}_{m,i}(s))-\lambda^{\Phi}_{m,i}(s))N^{\Phi}_{m,i}(\diff s)+\int_0^t(\sigma_i(s,\lambda^{\Phi}_{m,i}(s))-\lambda^{\Phi}_{m,i}(s))\diff s,
\end{split}
\end{equation}
where $(\lambda^{\Phi}_{m,i}(0))$ are random variables verifying Assumption \ref{Ass_2_cFIAP}. This is well-defined for the same reason \eqref{eq_RMF_cFIAP} is.

We formalize the law of large numbers we aim to prove as follows:
\begin{definition}
Let $M \in \N$. Let $(X^M_n)_{1 \leq n\leq M}$ be $M$-exchangeable random variables with finite expectation. We say they satisfy an $\mathcal{L}^1$ triangular law of large numbers, which we denote $\TLLN(X^M_n)$, if there exists $C>0$ such that
\begin{equation}
\label{eq_TLLN_cFIAP}
\E\left[\left|\frac{1}{M-1}\sum_{n=1}^M (X^M_n-\E[X^M_n])\right|\right]\leq \frac{C}{\sqrt{M}}
\end{equation}
and there exists a random variable $\tilde{X}$ and a constant $D>0$ such that for all $1\leq n_1 \neq n_2\leq M$ and for all $B_1, B_2\in \mathcal{B}(\R),$
\begin{equation}
\label{eq_TLLN_cFIAP_2}
|\P(X^M_{n_1} \in B_1, X^M_{n_2}\in B_2)-\P(\tilde{X} \in B_1)\P(\tilde{X} \in B_2)|\leq \frac{D}{\sqrt{M}}.
\end{equation}
\end{definition}

From \eqref{eq_chen_poisson_bound_cFIAP}, we know that if the triangular law of large numbers holds for the fixed point of $\Phi$, it allows for convergence in total variation of arrivals across replicas from a given neuron $j$ to a given neuron $i$ to a Poisson random variable. As such, our aim here is twofold:

\begin{enumerate}
\item Show that for all $(m,i)\in \{1,\ldots,M\}\times\{1,\ldots,K\}, \TLLN(\overline{N_{m,i}}([0,t]))$ implies \\$\TLLN(\Phi(\overline{N_{m,i}}([0,t])))$;
\item Show that $(\Phi^l)_{l \in \N^*}$ is a Cauchy sequence that converges to the fixed point.
\end{enumerate}
Since we can choose $\overline{N_{m,i}}([0,t])$ to be i.i.d. to ensure that there exist inputs for which the $\TLLN$ property holds, this will allow us to propagate the property and show that the $\TLLN$ property holds at the fixed point as well.

We will start by proving a lemma that will be key for the second point. The adaptation of this lemma from \cite{Davydov2022} to the more general framework is the main challenge in this part, but assumptions made on $f$ and Assumption \ref{Ass_main} allow this extension.

\begin{lemma}
\label{lem_sznitbound_cFIAP}
There exists $T>0$ such that for $\rho, \nu \in (\mathcal{P}(D_T))^{MK},$ there exists a constant $C_T>0$ such that
\begin{equation}
\label{eq_sznitbound_cFIAP}
K_{T,U}^{MK}(\Phi(\rho),\Phi(\nu))\leq C_T \int_0^T K_{t,U}^{MK}(\rho, \nu)\diff t.
\end{equation}
\end{lemma}
\begin{proof}
Let $T>0.$ Let $t\in[0,T]$. Fix $(m,i) \in \{1,\ldots,M\}\times\{1,\ldots,K\}$. Let $N^{\rho}$ (resp. $N^{\nu},N^{\Phi(\rho)},N^{\Phi(\nu)})$ be a point process admitting $\rho$ (resp. $\nu,\Phi(\rho),\Phi(\nu)$) as a stochastic intensity.
We have from \eqref{eq_input_output_map}
\begin{equation*}
\begin{split}
&\Phi(\rho)_{m,i}(t)-\Phi(\nu)_{m,i}(t)=\\
&f\left(\sum_{j\neq i}\sum_{n \neq m}\left(\int_0^t h_{j\rightarrow i}(s)\one_{\{V^M_{(n,j)\rightarrow i}(s)=m\}}(N^{\rho}_{n,j}(\diff s)-N^{\nu}_{n,j}(\diff s))\right)\right)\\
&+\int_0^t (g_i(s,\Phi(\rho)_{m,i}(s))-\Phi(\rho)_{m,i}(s))N^{\Phi(\rho)}_{m,i}(\diff s)+\int_0^t (\sigma_i(s,\Phi(\rho)(s))-\Phi(\rho)_{m,i}(s))\diff s\\
&-\int_0^t (g_i(s,\Phi(\nu)_{m,i}(s))-\Phi(\nu)_{m,i}(s))N^{\Phi(\nu)}_{m,i}(\diff s)-\int_0^t (\sigma_i(s,\Phi(\nu)(s))-\Phi(\nu)_{m,i}(s))\diff s.
\end{split}
\end{equation*}
\newpage
Let $(\hat{N}_{m,i})_{(m,i) \in \{1,\ldots,M\}\times\{1,\ldots,K\}}$ be independent Poisson point processes with intensity 1 on $[0,T]\times\R^+$. 
Using the Poisson embedding construction \eqref{eq_rmf_coupling_cFIAP}, we can write
\begin{equation*}
\begin{split}
&\Phi(\rho)_{m,i}(t)-\Phi(\nu)_{m,i}(t)=\\
&f\left(\sum_{j\neq i}\sum_{n \neq m}\int_0^t\int_0^{+\infty}h_{j\rightarrow i}(s)\one_{\{V^M_{(n,j)\rightarrow i}(s)=m\}}(\one_{\{u\leq \rho_{n,j}(s)\}}-\one_{\{u\leq \nu_{n,j}(s)\}})\hat{N}_{n,j}(\diff s\diff u)\right)\\
&+\int_0^t\int_0^{+\infty}g_i(s,\Phi(\rho)_{m,i}(s))\one_{\{ u\leq \Phi(\rho)_{m,i}(s)\}}-g_i(s,\Phi(\nu)_{m,i}(s))\one_{\{ u\leq \Phi(\nu)_{m,i}(s)\}}\hat{N}_{m,i}(\diff s\diff u)\\
&+\int_0^t\int_0^{+\infty}\Phi(\nu)_{m,i}(s)\one_{\{ u\leq \Phi(\nu)_{m,i}(s)\}}-\Phi(\rho)_{m,i}(s)\one_{\{ u\leq \Phi(\rho)_{m,i}(s)\}}\hat{N}_{m,i}(\diff s\diff u)\\
&+\int_0^{t}(\Phi(\nu)_{m,i}(s)-\Phi(\rho)_{m,i}(s))\diff s+\int_0^t(\sigma_i(s,\Phi(\rho)_{m,i}(s))-\sigma_i(s,\Phi(\nu)_{m,i}(s)))\diff s.
\end{split}
\end{equation*}

Therefore, using the fact that $f$ is Lipschitz and Assumption \ref{Ass_main}, we have the existence of a constant $D>0$ such that
\begin{equation*}
\begin{split}
&\left|\Phi(\rho)_{m,i}(t)-\Phi(\nu)_{m,i}(t)\right|\leq\\
&D\sum_{j\neq i}\sum_{n \neq m}\int_0^t\int_0^{+\infty}h_{j\rightarrow i}(s)\one_{\{V^M_{(n,j)\rightarrow i}(s)=m\}}\one_{\{u\leq \sup_{z \in [0,s]}|\rho_{n,j}(z)-\nu_{n,j}(z)|\}}\hat{N}_{n,j}(\diff s\diff u)\\
&+2\int_0^t\int_0^{+\infty}\sup_{z \in [0,s]}|\Phi(\rho)_{m,i}(z)-\Phi(\nu)_{m,i}(z)|\one_{\{u\leq \Phi(\rho)_{m,i}(s)\wedge\Phi(\nu)_{m,i}(s)\}}\hat{N}_{m,i}(\diff s\diff u)\\
&+2\int_0^t\int_0^{+\infty}|\Phi(\rho)_{m,i}(s)\vee\Phi(\nu)_{m,i}(s)|\one_{\{u\leq \sup_{z \in [0,s]}|\Phi(\rho)_{m,i}(z)-\Phi(\nu)_{m,i}(z)|\}}\hat{N}_{m,i}(\diff s\diff u)\\
&+2\int_0^{t}|\Phi(\nu)_{m,i}(s)-\Phi(\rho)_{m,i}(s)|\diff s.
\end{split}
\end{equation*}
Taking expectations of both sides and using the stochastic intensity property \eqref{eq_stoch_int_cfiap}, we get
\begin{equation*}
\begin{split}
&\E\left[\sup_{t \in [0,T]}\left|\Phi(\rho)_{m,i}(t)-\Phi(\nu)_{m,i}(t)\right|\right]\leq \\
&\frac{DH}{M-1}\sum_{j\neq i}\sum_{n \neq m}\int_0^T\E\left[\sup_{z\in[0,s]}\left|\rho_{n,j}(z)-\nu_{n,j}(z)\right|\right]\diff s\\
&+2\int_0^T\E\left[\sup_{z \in [0,s]}\left|\Phi(\rho)_{m,i}(z)-\Phi(\nu)_{m,i}(z)\right|\left(\Phi(\rho)_{m,i}(s)\wedge \Phi(\nu)_{m,i}(s)\right)\right]\diff s\\
&+2\int_0^T\E\left[\sup_{z \in [0,s]}\left|\Phi(\rho)_{m,i}(z)-\Phi(\nu)_{m,i}(z)\right|\left(\Phi(\rho)_{m,i}(s)\vee \Phi(\nu)_{m,i}(s)\right)\right]\diff s\\
&+2\int_0^T\E\left[\sup_{z \in [0,s]}\left|\Phi(\rho)_{m,i}(z)-\Phi(\nu)_{m,i}(z)\right|\right]\diff s.
\end{split}
\end{equation*}

We then have
\begin{equation*}
\frac{DH}{M-1}\sum_{j\neq i}\mu_{j\rightarrow i}\sum_{n \neq m}\int_0^T\E\left[\sup_{z\in[0,s]}\left|\rho_{n,j}(z)-\nu_{n,j}(z)\right|\right]\diff s\leq DH\sum_{j=1}^K\sum_{n=1}^M\int_0^T d_{D_s,U}(\rho_{n,j},\nu_{n,j})\diff s,
\end{equation*}
from which we immediately get by definition of $\mathcal{K}_{T,U}^{MK}$
\begin{equation}
\label{eq_contracterm_cFIAP}
\frac{DH}{M-1}\sum_{j\neq i}\sum_{n \neq m}\int_0^T\E[\sup_{z\in[0,s]}|\rho_{n,j}(z)-\nu_{n,j}(z)|]\diff s\leq DH\int_0^T \mathcal{K}_{s,U}^{MK}(\rho,\nu)\diff s.
\end{equation}

Let $C>0$. As before, let $A_C([0,T])=\{(\omega,t) \in \Omega\times [0,T], \Phi(\rho)_{m,i}(t)\vee \Phi(\nu)_{m,i}(t)>C\}$.
Using the exact same reasoning as in \cite[Lemma 2.12]{Davydov2022}, we have the existence of a constant $K_T>0$ such that
\begin{equation*}
\begin{split}
&\int_0^t\E[\sup_{z \in [0,s]}|\Phi(\rho)_{m,i}(z)-\Phi(\nu)_{m,i}(z)|(\Phi(\rho)_{m,i}(s)\wedge \Phi(\nu)_{m,i}(s))]\diff s\\
&\leq C\int_0^t \E\left[\sup_{z\in[0,s]}\left|\Phi(\rho)_{m,i}(z)-\Phi(\nu)_{m,i}(z)\right|\right]\diff s+K_Te^{-3CT}.
\end{split}
\end{equation*}
Plugging in \eqref{eq_contracterm_cFIAP} and applying the same reasoning as above to the last integral term, we get the existence of a constant $K'_T>0$ such that
\begin{equation*}
\begin{split}
&\E\left[\sup_{t \in [0,T]}\left|\Phi(\rho)_{m,i}(t)-\Phi(\nu)_{m,i}(t)\right|\right]\leq DH\int_0^T \mathcal{K}_{s,U}^{MK}(\rho,\nu)\diff s\\
&+\left(2(1+C)\right)\int_0^T\E\left[\sup_{z \in [0,s]}\left|\Phi(\rho)_{m,i}(z)-\Phi(\nu)_{m,i}(z)\right|\right]\diff s\\
&+\left(K_T+K'_T\right)e^{-3CT}.
\end{split}
\end{equation*}
Applying Grönwall's lemma, we get
\begin{equation*}
\begin{split}
&\E\left[\sup_{t \in [0,T]}\left|\Phi(\rho)_{m,i}(t)-\Phi(\nu)_{m,i}(t)\right|\right]\leq \\
&\left(DH\int_0^T \mathcal{K}_{s,U}^{MK}(\rho,\nu)\diff s + \left(K_T+K'_T\right)e^{-3CT}\right)e^{2(1+C)T}.
\end{split}
\end{equation*}
For any $\eps>0$, we can choose $C>0$ such that
\begin{equation*}
\E\left[\sup_{t \in [0,T]}\left|\Phi(\rho)_{m,i}(t)-\Phi(\nu)_{m,i}(t)\right|\right]\leq \left(DH\int_0^T \mathcal{K}_{s,U}^{MK}(\rho,\nu)\diff s\right)e^{2(1+C)T}+\eps.
\end{equation*}
Letting $\eps$ go to 0 and taking the sum over all coordinates and the infimum across all couplings, we get the result.
\end{proof}

As previously mentioned, we need to prove convergence of the sequence of iterates of $\Phi$ to the fixed point of $\Phi$ to prove the triangular law of large numbers. We will now derive this from Lemma \ref{lem_sznitbound_cFIAP}.

\begin{lemma}
\label{lem_iter_conv}
The sequence $(\Phi^l)_{l \in \N^*}$ of iterates of the function $\Phi$ is a Cauchy sequence. Moreover, it converges to the unique fixed point of $\Phi$.
\end{lemma}
We refer to \cite[Lemma 2.13]{Davydov2022} for the proof of this result, which is inspired by the classical approach of Sznitman, see \cite{Szn89}.

All that remains is proving that the triangular law of large numbers is carried over by the function $\Phi$, namely, that if we have some input $X$ that verifies $\TLLN(X)$, then we have $\TLLN(\Phi(X)).$

To do so, the key lemma will be the following law of large numbers. 

\begin{lemma}
\label{lem_tlln_cFIAP}
Let $M \in \N^*$. Let $(X_1^M, \ldots, X^M_M)$ be $M$-exchangeable centered random variables with finite exponential moments. Suppose that there exists $D>0$ such that for all $N \in \N^*$, all $B_1,B_2 \in \mathcal{B}(\R)$ and all $1\leq l,p\leq N$, $|\P(X^M_l \in B_1,X^M_p \in B_2)-\P(\tilde{X}_l\in B_1,\tilde{X}_p\in B_2)|\leq \frac{D}{\sqrt{M}}$, where $(\tilde{X}_i)_{i \in \N^*}$ are i.i.d. random variables and the convergence takes place in distribution. Then, there exists $C>0$ such that
\begin{equation}
\label{eq_tlln_cFIAP_lem}
\E\left[\left|\frac{1}{M}\sum_{n=1}^M X^M_n\right|\right] \leq \frac{C}{\sqrt{M}}.
\end{equation}

\end{lemma}
\begin{proof}
Let $$U_M=\frac{1}{M}\sum_{n=1}^M X^M_n.$$ Note that $\E[U_M]=0,$
and we have
\begin{equation*}
\E[U_M^2]=\frac{1}{M^2}\E\left[\sum_{n=1}^M(X^M_n)^2+\sum_{m,n=1,m \neq n}^M X^M_m X^M_n\right].
\end{equation*}
Since the exponential moments of $(X^M_n)$ are finite by Lemma \ref{lem_exp_moment_bound_cFIAP} and all the $X^M_n$ converge in distribution to $\tilde{X}^M_n$ by assumption, using dominated convergence, we have the existence of $Z_1>0$ such that for any $m,n$,
\begin{equation*}
|\E[X^M_m X^M_n]-\E[\tilde{X}_m \tilde{X}_n]|\leq \frac{Z_1}{\sqrt{M}}.
\end{equation*} 
By independence of $(\tilde{X}^M_n)$, $\E[\tilde{X}_m \tilde{X}_n]=\E[\tilde{X}_m]\E[\tilde{X}_n]=0$.

Therefore, by exchangeability, setting $Z_2=\sup_i\E[(X^M_i)^2]$, we have
\begin{equation*}
\E[U_M^2] \leq \frac{1}{M^2}(Z_2M+M(M-1)\frac{Z_1}{\sqrt{M}}).
\end{equation*}
Now, applying Chebychev's inequality, for any $\delta >0,$
\begin{equation*}
\P(|U_M-\E[U_M]|>\delta) \leq \frac{\E[U_M^2]}{\delta^2}\leq \frac{Z_2}{M\delta^2}+\frac{M(M-1)Z_1}{M^2\sqrt{M}\delta^2}.
\end{equation*}
This gives convergence in probability of $U_M$ with the desired speed when $M\rightarrow\infty$.

Since in addition the second moments are uniformly bounded, the desired result \eqref{eq_tlln_cFIAP_lem} follows.
\end{proof}

The following lemma is the last step needed to prove the main theorem:
\begin{lemma}
\label{lem_tlln_Phi_cFIAP}
Let $(\overline{N_{m,i}})$ be point processes on $[0,T]$ with finite exponential moments. Let $t \in [0,T].$ Suppose $\TLLN((\overline{N_{m,i}}([0,t]))$ holds. Then, \\
$\TLLN(\Phi((\overline{N_{m,i}}([0,t])))$ holds as well.
\end{lemma}
\begin{proof}
Most of the proof is analogous to that of \cite[Lemma 2.15]{Davydov2022}, the only differences coming from the fact that we now propagate convergence speed as well. We nevertheless write it out in full to introduce notation for the part that differs, namely the proof of the convergence in total variation of a sum knowing that each of the summands converges at the desired speed.

Suppose $\TLLN((\overline{N_{m,i}}([0,t]))$ holds. For $(m,i)\in \{1,\ldots,M\}\times\{1,\ldots,K\}$, let $\overline{\lambda_{m,i}}$ be the stochastic intensity of the process $\overline{N_{m,i}}$.
We write for all $t \in [0,T]$,
\begin{equation*}
\lambda_{m,i}(t)=\lambda_{m,i}(0)+\sum_{j \neq i}\overline{A_{j\rightarrow (m,i)}}(t)+\int_0^t (r_i-\lambda_{m,i}(s))N_{m,i}(\diff s),
\end{equation*}
where $(\lambda_{m,i}(0))$ verifies Assumption \ref{Ass_2_cFIAP} and 
\begin{equation*}
\overline{A_{j\rightarrow (m,i)}}(t)=\sum_{n \neq m}\int_0^t h_{j\rightarrow i}(s)\one_{\{V^M_{(n,j)\rightarrow i}(s)=m\}}\overline{N_{n,j}}(\diff s).
\end{equation*}

Analogously to \eqref{eq_chen_poisson_bound_cFIAP}, we have
\begin{equation*}
\begin{split}
&d_{TV}(\overline{A_{j\rightarrow (m,i)}}(t),\tilde{A}_{j \rightarrow i}(t)) \leq\\ &\left(1\wedge\frac{0.74}{\sqrt{\E[\overline{N_{1,j}}\left([0,T]\right)]}}\right)\frac{1}{M-1}\E[|\sum_{n \neq m}\left(\E[\overline{N_{n,j}}\left([0,T]\right)]-\overline{N_{n,j}}\left([0,T]\right)\right)|]\\
&+\frac{1}{M-1}\left(1\wedge \frac{1}{\E[\overline{N_{1,j}}\left([0,T]\right)]}\right)\E[\overline{N_{1,j}}\left([0,T]\right)].
\end{split}
\end{equation*}
where $\tilde{A}_{j \rightarrow i}$ are independent Poisson random variables with mean $\E[N_{1,j}([0,T])]$.

As such, $\TLLN((\overline{N_{m,i}}([0,t]))$ implies convergence in total variation of $(\overline{A_{j\rightarrow (m,i)}})$ to independent random variables $(\tilde{A}_{j \rightarrow i})$ with the desired $\sqrt{M}$ speed.

As in \cite{Davydov2022}, it remains to show that the convergence in total variation of $\sum_{j \neq i}\mu_{j\rightarrow i}\overline{A_{j\rightarrow (m,i)}}$ knowing that each of the summands converges at the desired speed.

Denote as before $\overline{N}=(\overline{N}_{n,j})_{n\neq m, j \neq i}.$ Let $q\in \N^{(M-1)(K-1)}.$ Let $B_1,B_2\in \mathcal{B}(\R^+).$ Let $l_1\neq l_2 \in \{1,\ldots,K\}\setminus \{i\}.$
Then, using the total probability formula, we have
\begin{equation*}
\begin{split}
&\P(\overline{A}_{l_1\rightarrow(m,i)}\in B_1,\overline{A}_{l_2\rightarrow(m,i)}\in B_2)=\\
&\sum_q \P(\overline{A}_{l_1\rightarrow(m,i)}\in B_1,\overline{A}_{l_2\rightarrow(m,i)}\in B_2|\overline{N}=q)\P(\overline{N}=q).
\end{split}
\end{equation*}
Using Lemma \ref{lem_cond_indep_bern_cFIAP}, by conditional independence, we have
\begin{equation*}
\begin{split}
&\P(\overline{A}_{l_1\rightarrow(m,i)}\in B_1,\overline{A}_{l_2\rightarrow(m,i)}\in B_2)=\\
&\sum_q \P(\overline{A}_{l_1\rightarrow(m,i)}\in B_1|\overline{N}=q)\P(\overline{A}_{l_2\rightarrow(m,i)}\in B_2|\overline{N}=q)\P(\overline{N}=q).
\end{split}
\end{equation*}
Using the same reasoning as in Lemma \ref{lem_poisson_chenstein_bound_cFIAP} and $\TLLN((\overline{N_{m,i}}))$, we have the existence of constants $C_1,C_2,C_3>0$ such that 
\begin{equation*}
|\P(\overline{A}_{l_1\rightarrow(m,i)}\in B_1|\overline{N}=q)-\P(\tilde{A}_{l_1\rightarrow i}\in B_1|\overline{N}=q)|\leq \frac{C_1}{\sqrt{M}},
\end{equation*} 
\begin{equation*}
|\P(\overline{A}_{l_2\rightarrow(m,i)}\in B_2|\overline{N}=q)-\P(\tilde{A}_{l_2\rightarrow i}\in B_2|\overline{N}=q)|\leq \frac{C_2}{\sqrt{M}}
\end{equation*} and 
\begin{equation*}
|\P(\overline{N}=q)-\P(\overline{\tilde{N}}=q)|\leq \frac{C_3}{\sqrt{M}}.
\end{equation*}
Moreover,
\begin{equation*}
\P(\tilde{A}_{l_1\rightarrow i}\in B_1|\overline{N}=q)=\P(\tilde{A}_{l_1\rightarrow i}\in B_1)
\end{equation*}
and
\begin{equation*}
\P(\tilde{A}_{l_2\rightarrow i}\in B_2|\overline{N}=q)=\P(\tilde{A}_{l_2\rightarrow i}\in B_2).
\end{equation*}
By dominated convergence, since

\begin{equation*}
\sum_q \P(\tilde{A}_{l_1\rightarrow i}\in B_1)\P(\tilde{A}_{l_2\rightarrow i}\in B_2)\P(\overline{N}=q)=\P(\tilde{A}_{l_1\rightarrow(m,i)}\in B_1,\tilde{A}_{l_2\rightarrow(m,i)}\in B_2),
\end{equation*}
we have the existence of a constant $D>0$ such that
\begin{equation*}
|\P(\overline{A}_{l_1\rightarrow(m,i)}\in B_1,\overline{A}_{l_2\rightarrow(m,i)}\in B_2)-\P(\tilde{A}_{l_1\rightarrow(m,i)}\in B_1,\tilde{A}_{l_2\rightarrow(m,i)}\in B_2)|\leq \frac{D}{\sqrt{M}}.
\end{equation*}
This implies convergence of $\sum_j \mu_{j\rightarrow i}\overline{A}_{j\rightarrow(m,i)}$ with the same convergence speed.
Finally, the mapping theorem implies convergence in total variation of $\lambda_{m,i}(t)$ when $M \rightarrow \infty$ with the same convergence speed, and, analogously, for couples $(\lambda_{m,i}(t),\lambda_{m',i}(t))$. All conditions of Lemma \ref{lem_tlln_cFIAP} are thus satisfied. Applying it completes the proof.   
\end{proof}
Thus, we can now state the result that we were aiming to prove:
\begin{lemma}
Denote by $(N_{m,i})$ the point processes of the $M$-replica RMF dynamics \eqref{eq_RMF_cFIAP} that are the fixed point of $\Phi$. Then $\TLLN((N_{m,i}([0,T])))$ holds.
\end{lemma}
The proof is identical to \cite[Lemma 2.16]{Davydov2022}, replacing limits in $M$ by inequalities.

\section{Link between RMF FIAPs in discrete and continuous times~: the example of the excitatory Galves-Löcherbach model}
\label{sec_RMF_GL_to_RMF_FIAP}
FIAPs, and replica-mean-field versions of FIAPs, were originally introduced in discrete time in \cite{bdt_2022}. One natural question is to explore the links between these original FIAPs and the cFIAPS introduced in this work. The goal of this section is to show a link between replica-mean-field versions of continuous-time and discrete-time FIAPs for a specific instance of FIAP~: the excitatory Galves-Löcherbach model. In this particular case, the goal is to prove the existence of the horizontal equivalences in the following diagram:
\[\xymatrix{
  (\textrm{infinite-replica FIAP with }\textrm{-time step})_{\delta >0}  \ar[r] & \textrm{infinite-replica GL cFIAP} \ar[l] \\
 ( M\textrm{-replica FIAP with } \delta \textrm{-time step})_{\delta >0} \ar[u]^{M \rightarrow \infty} \ar[r] & M\textrm{-replica excitatory GL cFIAP} \ar[u]^{M \rightarrow \infty} \ar[l]}
\]
The left up arrow corresponds to the proof of the Poisson Hypothesis for a collection of discrete-time FIAPs as introduced in \cite{bdt_2022} with time-step $\delta$, for all $\delta>0$. The right up arrow corresponds to the proof of the Poisson Hypothesis for the cFIAP excitatory GL model, which was presented in \cite{Davydov2022}. In this section we complete the diagram by showing that it is possible to construct the discrete-time RMF FIAPs given the continuous-time dynamics, and vice-versa.

We recall the definition of discrete-time FIAPs originally introduced in \cite{bdt_2022}:
\begin{definition}
\label{def_FIAP_recall}
An instance of the class $\Cset$ of discrete
\textit{fragmentation-interaction-aggregation processes} 
is determined by:
\begin{itemize}
\item An integer $K$ representing the number of nodes;
\vspace{-5pt}
\item A collection of initial conditions for the integer-valued state variables at step zero, which we denote by $\{X_{i}\}$, where $i \in \{1,\ldots,K\}$; 
\vspace{-5pt}
\item A collection of fragmentation random variables $\{U_{i}\}$, which are i.i.d. uniform in $[0,1]$ and independent from $\{X_{i}\}$, where $i \in \{1,\ldots,K\}$; 
\vspace{-5pt}
\item A collection of {\em fragmentation functions} $\{g_{1,i} : \mathbb{N} \rightarrow \mathbb{N}\}_{i \in \{1,\ldots,K\}}$ \\
and $\{g_{2,i} : \mathbb{N} \rightarrow \mathbb{N}\}_{i \in \{1,\ldots,K\}}$; 
\vspace{-5pt}
\item A collection of bounded {\em interaction functions} $\{h_{j\rightarrow i} : \mathbb{N} \rightarrow \mathbb{N}\}_{i,j \in \{1,\ldots,K\}}$; 
\vspace{-5pt}
\item A collection of {\em activation probabilities} $\{\sigma_i(0), \sigma_i(1),\ldots\}_{i\in \{1,\ldots,K\}}$ verifying the 
conditions $\sigma_i(0)=0$,  and
$0<\sigma_i(1)\le \sigma_i(2)\le \cdots\le 1$ for all $i$.
\end{itemize}
The associated dynamics take as input 
the initial integer-valued state variables $\{X_{i}\}$
and define the state variables at the next step as
\begin{equation}
\label{eq_gendef_1_recall}
Y_{i}=g_{1,i}(X_{i})\one_{\{U_{i}<\sigma_i(X_{i})\}}+g_{2,i}(X_{i})\one_{\{U_{i}>\sigma_i(X_{i})\}}+A_{i}, \quad \forall i=1,\ldots,K,
\end{equation}
with arrival processes 
\begin{equation}
\label{eq_gendef_2_recall}
A_{i}=\sum_{j \neq i} h_{j\rightarrow i}(X_{j})\one_{\{U_{j}<\sigma_j(X_{j})\}},
\quad \forall i=1,\ldots,K.
\end{equation}
\end{definition}
Given a FIAP, we now consider its replica-mean-field model. We once again recall the precise definition from \cite{bdt_2022}:
\begin{definition}
For any process in $\Cset$, the associated $M$-replica dynamics is entirely specified by 
\begin{itemize}
\item A collection of initial conditions for the integer-valued state variables at step zero, which we denote by $\{X^{M}_{n,i}\}$, where $n \in \{1,\ldots,M\}$ and $i \in \{1,\ldots,K\}$, such that for all $M, n$ and $i$, $X^M_{n,i}=X_i$; 
\vspace{-5pt}
\item A collection of fragmentation random variables $\{U_{n,i}\}$, which are i.i.d. uniform in $[0,1]$ and independent from $\{X^{M}_{n,i}\}$, where $n \in \{1,\ldots,M\}$ and $i \in \{1,\ldots,K\}$; 
\vspace{-5pt}
\item A collection of i.i.d. \textit{routing} random variables $\{R^M_{(n,j)\rightarrow i}\}$ independent from $\{X^{M}_{n,i}\}$ and $\{U_{n,i}\}$, uniformly distributed on  $\{1,\ldots,M\}\setminus\{m\}$ for all $i,j \in \{1,\ldots,K\}$ and $m \in \{1,\ldots,M\}$. In other words, if $R^M_{(n,j)\rightarrow i}=n$, then an eventual activation of node $j$ in replica $m$ at step 0 induces an arrival of size $h_{j\rightarrow i}(X^M_{m,j})$ in node $i$ of replica $n$, and $n$ is chosen uniformly among replicas and independently from the state variables. Note that these variables are defined regardless of the fact that an activation actually occurs.
Also note that for $i'\ne i$, the activation in question will 
induce an arrival in node $i'$ of replica $n'$, with $n'$
sampled in the same way but independently of $n$.
\end{itemize}
Then, the integer-valued state variables at step one, denoted by $\{Y^{M}_{n,i}\}$, are given by the $M$-RMF equations
\begin{equation}
\label{eq_gen_rmf_1_recall}
Y^{M}_{n,i}=g_{1,i}(X^{M}_{n,i})\one_{\{U_{n,i}<\sigma_i(X^{M}_{n,i})\}}+g_{2,i}(X^{M}_{n,i})\one_{\{U_{n,i}>\sigma_i(X^{M}_{n,i})\}}+A^M_{n,i},
\end{equation}
where $g_{1,i}$, $g_{2,i}$ denotes fragmentation functions, $\sigma_i$ denotes activation probabilities, and where
\begin{equation}
\label{eq_gen_rmf_2_recall}
A^M_{n,i}=\sum_{m \neq n}\sum_{j \neq i} h_{j\rightarrow i}(X^M_{m,j})\one_{\{U_{m,j}<\sigma_i(X^{M}_{m,j})\}}\one_{\{R^M_{(n,j)\rightarrow i}=n\}}
\end{equation}
is the number of arrivals to node $i$ of replica $n$ via the interaction functions $h_{j\rightarrow i}$. 
\end{definition}

We will focus here on the case of the excitatory Galves-Löcherbach model (we will omit excitatory hereafter), in both discrete and continuous time settings, which we will now recall.

In the continuous time setting, the $M$-replica-mean-field of the Galves-Löcherbach model is defined as follows:
\begin{equation}
\label{eq_RMF_GL_SDE}
\begin{split}
\lambda^M_{m,i}(t)&=\lambda^M_{m,i}(0)+\frac{1}{\tau_i}\int_0^t\left(b_i-\lambda^M_{m,i}(s)\right)\diff s \\
&+\sum_{j \neq i}\mu_{j\rightarrow i}\sum_{n \neq m} \int_0^t \one_{\{V^M_{(n,j)\rightarrow i}(s)=m\}} N^M_{n,j}(\diff s)+\int_0^t \left(r_i-\lambda^M_{m,i}(s)\right) N^M_{m,i}(\diff s).
\end{split}
\end{equation}

In the discrete time setting with time step length $\delta$, the one-step transition of the RMF GL FIAP is given by
\begin{equation}
\label{eq_rmf1_recall}
Y^{M}_{m,i}=\one_{\{U_{m,i}>\sigma_{\delta}(X^{M}_{m,i})\}} X^{M}_{m,i}+\one_{\{U_{m,i}>\sigma_{\delta}(X^{M}_{m,i})\}}r_i+A^M_{m,i},
\end{equation}
where
\begin{equation}
\label{eq_rmf2_recall}
A^M_{m,i}=\sum_{n \neq m}\sum_{j \neq i} \mu_{j\rightarrow i}\one_{ \{U_{n,j}<\sigma_{\delta}(X^{M}_{n,j})\}}\one_{\{R^M_{n,j,i}=m\}}
\end{equation}
is the number of arrivals to neuron $i$ of replica $m$, $X^M_{m,i}$ is the state of neuron $i$ in replica $m$ at time 0 and $Y^M_{m,i}$ is its state at time 1.

In FIAP models, the arrivals to a given neuron at a given time are conditionally independent from the spiking activity of that neuron given the states at that time. Since in a continuous-time model all events are asynchronous, such a property is no longer verified. Thus, in order to map the continuous-time model to a FIAP, we must ``separate'' the arrivals and the spikes. In order to do that, we introduce a $\delta >0 $ unit of time. We then show that one can construct a discrete-time Markov chain that is similar to the embedded Markov chain of the continuous time model but which belongs to the RMF FIAP class. 

Since in FIAPs, all the states taken by the state variables are discrete, we must make the following simplifying assumption:

\begin{assumption}
\label{Ass_3}
\hspace{0.5cm}
\begin{itemize}
\item For all $i \in  \{1,\ldots,K\}$, $\tau_i=\infty$ (no exponential decay) and $r_i \in \N^*$;
\item For all $i,j \in \{1,\ldots,K\}$, $\mu_{j\rightarrow i} \in \N$.
\end{itemize}
\end{assumption}

Under Assumption \ref{Ass_3}, it is known (see \cite{Baccelli_2019}) that the generator of the $M$-replica dynamics is given by 

\begin{equation*}
\mathcal{A}[f](\mathbf{\lambda})=\sum_{i=1}^K\sum_{n=1}^M \frac{1}{|\mathcal{V}_{m,i}|}\sum_{v \in \mathcal{V}_{m,i}} \left( f(\mathbf{\lambda}+\mathbf{\mu}_{m,i,v}(\mathbf{\lambda}))-f(\mathbf{\lambda})\right)\lambda_{m,i},
\end{equation*}
where the update due to the spiking of neuron $(m,i)$ is defined by
\begin{equation*}
\left[\mathbf{\mu}_{m,i,v}(\mathbf{\lambda})\right]_{n,j}= 
\begin{cases}
\mu_{j \rightarrow i} & \text{ if } j \neq i, n=v_j \\
r_i - \lambda_{m,i} & \text{ if } j=i,n=m \\
0 & \text{ otherwise }.
\end{cases}
\end{equation*}

We now consider the embedded discrete-time Markov chain of the RMF GL dynamics, where all updates happen at the spiking times of the dynamics.
Since the spiking times are all distinct, all the transitions of the embedded Markov chain are given by
\begin{equation*}
\xymatrix{
(\lambda_{1,1},...,\lambda_{1,K},\lambda_{2,1},\ldots,\lambda_{M,K}) \ar[d]\\ 
(\lambda_{1,1},\ldots,\lambda_{m_1,1}+\mu_{i,1},\ldots,\lambda_{m,i-1},r_i,\lambda_{m,i+1},\ldots,\lambda_{M,K})
}
\end{equation*}
with associated transition probabilities $p^M_{m,i}\left(\frac{1}{M-1}\right)^{K-1}$,
where $p^M_{m,i}$ is the probability that neuron $(m,i)$ spikes conditioned on the event that a spike happens. The main complexity with this Markov chain is that the transition times correspond to the spiking times of the RMF GL network, which are not tractable.

Therefore, we now define a new discrete time Markov chain on $\N^{MK}$ with steps in time of fixed length $\delta$. The informal idea is to reset all neurons that spike in the RMF GL dynamics during such a step in time and update all the states with the potential due to these spikes. Since a single neuron could very well spike multiple times in a $\delta$ unit of time, we only consider the updates due to the first spike of a given neuron.

Note that the informal idea given above is the motivation behind the definition of the following discrete time dynamics, which is defined \textit{ad hoc}. We characterize the dynamics of this chain through its transitions, and we will show that the chain we define belongs to the class of discrete time RMF FIAPs.

In order to simplify notation and facilitate understanding, let us define the following ``half-step'' fragmentation and aggregation transitions $\mathcal{P}^A(\delta)$ and $\mathcal{P}^F(\delta)$.
We define $\mathcal{P}^F(\delta)$ as the transitions
\begin{equation*}
\xymatrix{
\mathcal{P}^F(\delta):(\lambda_{1,1},\ldots,\lambda_{1,K},\lambda_{2,1},\ldots,\ldots,\lambda_{M,K}) \ar[d]\\ 
\left(\lambda_{1,1},\ldots,r_{i_1},\ldots,r_{i_2},\ldots,r_{i_L},\ldots,\lambda_{M,K}\right)
}
\end{equation*}

where the transition probability is given by
\begin{equation*}
\prod_{(m,i) \in \mathcal{L}} \ p^M_{m,i}(\delta) \prod_{(m,i) \notin \mathcal{L}} \left(1-p^M_{m_k,i_k}(\delta) \right),
\end{equation*}
where $p^M_{m,i}(\delta)$ is the probability that neuron $(m,i)$ spikes, i.e., is set to the value $r_i$, in a $\delta$-unit of time and $\mathcal{L}=\{(m_1,i_1),\ldots,(m_L,i_L)\}.$ 

Informally, this transition corresponds to a \textit{fragmentation} of the state: we reset to their base rate all the neurons that spike during the $\delta$ step of time. Note that $p^M_{m,i}(\delta)=1-e^{-\sigma(\lambda_{m,i})\delta}$, where for all $k$, $\sigma(k)$ is the probability that a neuron of the RMF GL network in state $k$ spikes in a unit of time in length 1.

We then define the \textit{aggregation} transitions $\mathcal{P}^A(\delta)$ in the following fashion~: for all $1 \leq k \leq L$, each neuron $(m_k,i_k)$ which has spiked and been reset to its base rate $r_{i_k}$ in the previous step, for each $j \neq i_k$, we randomly, uniformly and independently from each other and from the rest of the dynamics, choose an index $n_k \neq m_k$ and increment neuron $(n_k,j)$ by $\mu_{i_k,j}$. Thus, all transitions are of the form
\begin{equation*}
\xymatrix{
\mathcal{P}^A(\delta):\left(\mathcal{P}^F(\delta)\left(\lambda_{1,1},\ldots,\lambda_{M,K}\right)\right) \ar[d]\\ 
\left(\lambda_{1,1},\ldots,\lambda_{n_{i_1},1}+\mu_{i_1,1},\ldots,\lambda_{n_{i_k,j}}+\mu_{i_k,j},\ldots,\lambda_{M,K}\right)
}
\end{equation*}
with transition probability $\left(\frac{1}{M-1}\right)^{(K-1)L}$ (conditioned on the transition probability of $\mathcal{P}^F(\delta))$. Note that $(n_{i_k})$ are not necessarily distinct for $k' \neq k$. Note that all the routings are done independently from one another and that updates to a single neuron are independent of whether that particular neuron has spiked or not.

We can then define the full transitions of our new discrete time Markov chain as 
\begin{equation}
\label{eq_delta_markov}
\left(\lambda_{1,1},\ldots,\lambda_{M,K}\right) \longmapsto \mathcal{P}^A(\delta)\circ \mathcal{P}^F(\delta)\left(\lambda_{1,1},\ldots,\lambda_{M,K}\right).
\end{equation}
By the total probability formula, the transition probability is given by 
\begin{equation}
\label{eq_delta_fiap_trans}
\sum_{l=1}^{MK}\sum_{\substack{J \subset \{1,\ldots,M\}\times \{1,\ldots,K\}\\ |J|=l}}\prod_{(n,j) \in J}p^M_{n,j}(\delta)\prod_{(n,j) \notin J}\left(1-p^M_{n,j}(\delta)\right)\left(\frac{1}{M-1}\right)^l.
\end{equation}
Our goal is now to show that the Markov chain defined this way belongs to the class of RMF FIAPs. In order to do that, we will require the following lemma, which gives the transition probability of a single coordinate of the above Markov chain.

\begin{lemma}
\label{Lem_trans_RMF_delta}
Let $k,l \in \N.$ Let $P^{M,m,i}_{k \rightarrow l}(\delta)$ be the probability that neuron $(m,i)$ following the RMF dynamics given above in \eqref{eq_delta_markov} transitions from state $k$ to state $l$ in a single step of time of length $\delta$.
Then
\begin{equation}
\label{eq_trans_RMF_delta_1}
P^{M,m,i}_{k \rightarrow l}=p^M_{m,i}(\delta)\P(A^M_{m,i}=l-r_i)+\left(1-p^M_{m,i}(\delta)\right)\P(A^M_{m,i}=l-k),
\end{equation}
where for all $l \in \N$,
\footnotesize
\begin{equation}
\label{eq_trans_RMF_delta_2}
\begin{split}
\P(A^M_{m,i}=l)&=\sum_{p=l}^{(K-1)(M-1)}\sum_{\substack{J \subset \{1,\ldots,M\}\setminus\{m\}\times \{1,\ldots,K\}\setminus\{i\}\\
\sum_{(n,j) \in J}\mu_{j\rightarrow i} =p}} \prod_{(n,j) \in J}p^M_{n,j}(\delta)\\
&\prod_{(n,j) \notin J}\left(1-p^M_{n,j}(\delta)\right)\left(\frac{1}{M-1}\right)^l\left(1-\frac{1}{M-1}\right)^{p-l}.
\end{split}
\end{equation}
\normalsize
Here $A^M_{m,i}$ are the updates due to arrivals to neuron $(m,i)$ in a $\delta$-unit of time.
\end{lemma}
\begin{proof}
Equation \eqref{eq_trans_RMF_delta_1} is due to the independence between the spiking of $(m,i)$ and the arrivals to $(m,i)$.
Let $S^M_{m,i}$ be the arrivals caused by spikes in the system in a $\delta$-unit of time, discounting the spikes in replica $m$ and the spikes in neuron $i$ across replicas. Informally, $S^M_{m,i}$ is the quantity from spikes that could potentially reach neuron $i$ in replica $m$ if the routing variables allow it.
Then 
\begin{equation}
\label{eq_tech_1}
\P(S^M_{m,i}=p)=\sum_{\substack{J \subset \{1,\ldots,M\}\setminus\{m\}\times \{1,\ldots,K\}\setminus\{i\}\\
\sum_{(n,j) \in J}\mu_{j\rightarrow i} =p}}\prod_{(n,j) \in J}p^M_{n,j}(\delta)\prod_{(n,j) \notin J}\left(1-p^M_{n,j}(\delta)\right).
\end{equation}
By the total probability formula,
\begin{equation}
\label{eq_tech_2}
\P(A^M_{m,i}=l)=\sum_{p=l}^{(M-1)(K-1)}\P(S^M_{m,i}=p)\P(A^M_{m,i}=l|S^M_{m,i}=p).
\end{equation}
Since the routings are independent of the rest of the process, we have
\begin{equation}
\label{eq_tech_3}
\P(A^M_{m,i}=l|S^M_{m,i}=p)=\left(\frac{1}{M-1}\right)^l\left(1-\frac{1}{M-1}\right)^{p-l}.
\end{equation}
Putting together \eqref{eq_tech_1},\eqref{eq_tech_2} and \eqref{eq_tech_3}, we get \eqref{eq_trans_RMF_delta_2}.
\end{proof}

Now we prove the following lemma:
\begin{lemma}
\label{Lem_corresp_RMF_FIAP}
The discrete-time Markov chain with transitions defined by \eqref{eq_delta_markov} is a RMF FIAP.
\end{lemma}
\begin{proof}
In order to achieve that, we compute the transition probabilities of RMF FIAPs and we show that the transition probabilities of the Markov chain defined above are of that type.

Consider the RMF GL FIAP model following dynamics given by \eqref{eq_rmf1_recall} and \eqref{eq_rmf2_recall}. Let $X=\{X^M_{m,i}\}$ be the state variables at step 0, let $\sigma_{\delta} : \N \rightarrow [0,1]$ be the spiking probabilities of the neuron satisfying the conditions given in Definition \ref{def_FIAP_recall}.

We now give the transition probabilities of the Markov RMF FIAP dynamics.
Let $k,l \in \N.$ Let $Q^{M,m,i}_{k \rightarrow l}$ be the probability that neuron $(m,i)$ in the RMF FIAP dynamics is in state $l$ at time $\delta$ given that it is in state $k$ at time 0. In other words, $Q^{M,m,i}_{k \rightarrow l}=\P(Y^M_{m,i}=l|X^M_{m,i}=k)$.
Then
\begin{equation}
\label{eq_trans_RMF_FIAP_1}
Q^{M,m,i}_{k \rightarrow l}=\sigma_{\delta}(X^M_{m,i})\P(A^M_{m,i}=l-r_i)+\left(1-\sigma_{\delta}(X^M_{m,i})\right)\P(A^M_{m,i}=l-k),
\end{equation}
where
\footnotesize{
\begin{equation}
\label{eq_trans_RMF_FIAP_2}
\begin{split}
\P(A^M_{m,i}=l)&= \sum_{p=l}^{(K-1)(M-1)}\sum_{\substack{J \subset \{1,\ldots,M\}\setminus\{m\}\times \{1,\ldots,K\}\setminus\{i\}\\
\sum_{(n,j) \in J}\mu_{j\rightarrow i} =p}}\prod_{(n,j) \in J}\sigma_{\delta}(X^M_{n,j})\\
&\prod_{(n,j) \notin J}\left(1-\sigma_{\delta}(X^M_{n,j})\right)\left(\frac{1}{M-1}\right)^l\left(1-\frac{1}{M-1}\right)^{p-l}.
\end{split}
\end{equation}
}
\normalsize
Here $A^M_{m,i}$ are the updates due to arrivals to neuron $(m,i)$ in a $\delta$-unit of time.

Note that in the case where $\mu_{j\rightarrow i}=1$ for all $i,j$, $A^M_{m,i}$ represents the number of arrivals to neuron $(m,i)$ in a $\delta$-unit of time.

Equation \eqref{eq_trans_RMF_FIAP_1} is simply due to the independence of arrivals and spikes in the RMF FIAP model. We now proceed identically to the proof of the particular case above. Let $S^M_{m,i}$ be the arrivals caused by spikes in the system in a $\delta$-unit of time, discounting the spikes in replica $m$ and the spikes in neuron $i$ across replicas. 
Then 
\begin{equation}
\label{eq_tech_4}
\P(S^M_{m,i}=p)=\sum_{\substack{J \subset \{1,\ldots,M\}\setminus\{m\}\times \{1,\ldots,K\}\setminus\{i\}\\
\sum_{(n,j) \in J}\mu_{j\rightarrow i} =p}}\prod_{(n,j) \in J}\sigma(X^M_{n,j})\prod_{(n,j) \notin J}\left(1-\sigma(X^M_{n,j})\right).
\end{equation}
By the total probabilities formula,
\begin{equation}
\label{eq_tech_5}
\P(A^M_{m,i}=l)=\sum_{p=l}^{(M-1)(K-1)}\P(S^M_{m,i}=p)\P(A^M_{m,i}=l|S^M_{m,i}=p).
\end{equation}
Since the routing variables are independent of the rest of the process, we have
\begin{equation}
\label{eq_tech_6}
\P(A^M_{m,i}=l|S^M_{m,i}=p)=\left(\frac{1}{M-1}\right)^l\left(1-\frac{1}{M-1}\right)^{p-l}.
\end{equation}
Putting together \eqref{eq_tech_4},\eqref{eq_tech_5} and \eqref{eq_tech_6}, we get \eqref{eq_trans_RMF_FIAP_2}.

As such, we see that the transition probabilities in our Markov chain model given by \eqref{eq_trans_RMF_delta_1} and \eqref{eq_trans_RMF_delta_2} are a particular case of these general FIAP transition probabilities\eqref{eq_trans_RMF_FIAP_1} and \eqref{eq_trans_RMF_FIAP_2}. This concludes the proof.
\end{proof}

In this way, we have shown that given a RMF GL continuous-time model, given the initial conditions, we can uniquely define a collection of RMF FIAP discrete-time models with varying time step lengths associated with it.

Note that a reverse construction is also possible in the following sense: given RMF FIAP dynamics of the type defined above for all $\delta >0$, since for all $k$, $\sigma_{\delta}(k)=\delta\sigma(k)+o(\delta)$, we can reconstruct the infinitesimal generator of the continuous-time dynamics by considering the transition operator $\frac{1}{\delta}\left(P_{\delta}-Id\right)$, where $Id$ is the identity operator and $P_{\delta}$ is the transition operator of the RMF FIAP dynamics with time steps of length $\delta$, and letting $\delta$ go to 0.



\begin{acks}[Acknowledgments]
The author would like to thank François Baccelli for his guidance and suggestions and Kavita Ramanan for her comments and suggestions.
\end{acks}

\begin{funding}
The author was supported by the ERC NEMO grant (\# 788851) to INRIA Paris and by the Office of Naval Research under the Vannevar Bush Faculty Fellowship N0014-21-1-2887. This work was supported by the CRCNS program of the National Science Foundation under award number DMS-2113213. The author acknowledges support of the Institut Henri Poincaré (UAR 839 CNRS-Sorbonne Université), and LabEx CARMIN (ANR-10-LABX-59-01).
\end{funding}
\newpage
\bibliography{biblio_cFIAP}
\bibliographystyle{abbrv}

\end{document}